\documentclass[10pt,a4paper,reqno]{amsart}
\usepackage{amssymb}
\usepackage{amsmath}
\usepackage{amssymb,latexsym}
\usepackage[dvips]{graphicx}
\usepackage{color}
\usepackage{cite}
\usepackage{amsthm}
\usepackage{cancel}
\usepackage[normalem]{ulem}
\usepackage{soul}
\newcommand{\R}{\mathbb R}
\newcommand{\Z}{\mathbb Z}
\newcommand{\N}{\mathbb N}

\newcommand{\sgn}{\rm sign}

\newtheorem{theorem}{Theorem} [section]
\newtheorem{lemma}{Lemma} [section]
\newtheorem{proposition}{Proposition} [section]

\newtheorem{definition}{Definition} [section]

\newtheorem{remark}{Remark}[section]

\let\ssection=\section\renewcommand{\section}{\setcounter{equation}{0}\ssection}

\newcommand{\tendsto}[1]{\renewcommand{\arraystretch}{0.5}
\begin{array}[t]{c}
\longrightarrow \\
{ \scriptstyle #1 }
\end{array}
\renewcommand{\arraystretch}{1}}

\newcommand{\weaktendsto}[1]{\renewcommand{\arraystretch}{0.5}
\begin{array}[t]{c}
\rightharpoonup \\
{ \scriptstyle #1 }
\end{array}
\renewcommand{\arraystretch}{1}}
\newcommand{\norm}[2]{\big\| #1 \big\| _{#2}}
\newcommand{\cro}[1]{\langle #1 \rangle}

\title{The classical Boussinesq system revisited}
\author[Luc Molinet]{Luc Molinet$^1$}
\thanks{ $^1$Institut Denis Poisson, Universit\'e de Tours, Universit\'e d'Orl\'eans, CNRS, Parc Grandmont, 37200, France ({\tt  luc.molinet@univ-tours.fr}).}
\author[Raafat Talhouk]{Raafat Talhouk$^2$}  
\author[Ibtissam Zaiter]{Ibtissam Zaiter$^2$}
\thanks{ $^2$Laboratoire de math\'ematiques-EDST, Facult\'e des Sciences et EDST, Universit\'e Libanaise, Hadat, Liban 
{\tt (rtalhouk@ul.edu.lb), (zibtissame@hotmail.com)}}

\begin{document}

\maketitle
\begin{abstract}
In this work, we revisit the study by M. E. Schonbek \cite{Sch} concerning the problem of existence of global entropic weak solutions for the classical Boussinesq system, as well as the study of the regularity of these solutions by C. J. Amick \cite {Ami}. We propose to regularize by  a "fractal" operator (i.e. a differential operator defined by a Fourier multiplier of type $\epsilon  | \xi | ^ {\lambda}, \,
 (\epsilon,\lambda) \in\,\R_+\times ] 0,2] $). We first show that the regularized system is globally unconditionally well-posed in Sobolev spaces of type $ H^s (\R), \, s> \frac {1} {2}, $ uniformly in the regularizing parameters $(\epsilon,\lambda) \in\,\R_+\times ] 0,2] $. As a consequence we obtain the   global well-posedness of the classical Boussinesq system at this level of regularity as well as the convergence in the strong topology of the solution of the regularized system towards the solution of the classical Boussinesq equation as the parameter $ \epsilon $ goes to $ 0 $.
 In a second time, we prove  the existence of low regularity entropic solutions  of the Boussinesq equations emanating from $ u_0 \in H^1 $  and $ \zeta_0 $ in an Orlicz class   as weak limits of  regular solutions. 
\end{abstract}
\section{Introduction}
In this paper we are concerned with the classical Boussinesq system, introduced by J. V. Boussinesq in 1871 to describe weak amplitude long wave propagation on the surface of ideal incompressible liquid for irrotational flow submitted to  gravitational force where the surface tension has been neglected. In 2002, Bona, Chen and Saut \cite{BCS02} have derived a class
of models called four parameters Boussinesq systems. The corresponding PDE's system is given by:

\begin{eqnarray}\label{abcd}
\left\{
\begin{array}{lcl}
\zeta_t+u_x+(u\zeta)_x + au_{xxx}-b\zeta_{xxt}&=&0,\\
u_t+\zeta_x+uu_x+c\zeta_{xxx}-du_{xxt}&=&0.
\end{array} \right.
\end{eqnarray}

$\zeta(x,t)+1$ corrrespond to the normalized total height of the liquid and then describe the free surface of the liquid, $x$ is the spatial position which is proportional to distance in the direction of propagation. $u(x,t)$ is the horizontal velocity field of the liquid particle which is at position $x$ at time $t.$ $a,b,c$ and $d$ are  four parameters verifying consistence relation (see \cite{BCS02}). The classical Boussinesq system corresponds to the choice of parameters: $a=b=c=0$ and $d=1$ and the system becomes:
 
\begin{eqnarray}\label{1}
\left\{
\begin{array}{lcl}
\zeta_t+u_x+(u\zeta)_x &=&0,\\
u_t+\zeta_x+uu_x-u_{xxt}&=&0.
\end{array} \right.
\end{eqnarray}

Schonbek (in \cite{Sch}) have shown the existence of global in time weak solution under a natural non-cavitation condition ($1+\zeta_0>0$) with initial data $\zeta_0$ in some Orlicz class and $u_0\in H^1(\R).$ She used a viscosity method by regularizing the first equation with the Laplace operator after what a uniform entropic estimate is established. This entropic estimate allowed to passing to the limit and defining a weak solution for the classical Boussinesq system. Amick (in \cite{Ami}) showed that weak solutions given by Schonbek are in fact infinitely regular, i.e. in $C_0^{\infty}$ if the initial data are $C^{\infty}_c.$ Actually the results of Amick are implicitly containing also that the entropic solution is in $H^k$  if the initial data are in classical regular spaces of type $ \bold {H^k \times H^{k+1}, \forall k \in \N,\, k\geq 2}.$ Bona \& all. (in \cite{BCS04}) studied many cases of giving $a,\,b,\, c,\, d$ parameters
and in particular concerning system (\ref{1}) they give, without proof, existence and uniqueness results of solution  $ (\zeta,u)\in C([0,T]; H^s\times H^{s+1})$ for given initial data in $H^s\times H^{s+1},\; s\geq1$ with $\inf_{x\in\R}(1+\zeta_0(x))>0$ and announcing the continuity of the flow on more restricted class of initial data. All the previous studies are in one dimension, many other studies of the four parameters Boussinesq system in the last ten years concerning the two dimensional case, see for instance \cite{SWX} and references therein.  

In our work we reconsider the method of regularization by using generalized derivative operator, also called "fractal" operator, that is a  differential operator defined by a Fourier multiplier of type $ | \xi | ^ {\lambda}, \, \lambda \in ]0,2]. $ More precisely we consider the following regularized system:

\begin{eqnarray}\label{2}
\left\{
\begin{array}{lcl}
\zeta_t+u_x+(u\zeta)_x+ \epsilon g_\lambda(\zeta) &=&0,\\
u_t+\zeta_x+uu_x-u_{xxt}&=&0.
\end{array} \right.
\end{eqnarray}
where $g$ is the non-local operator defined through the Fourier transform by
\begin{equation}\label{op}
 \mathcal{F}(g[\varphi(t,\cdot)])(\xi)= |\xi|^\lambda \mathcal{F}(\varphi(t,\cdot))(\xi),  \;\; \mbox{ with } \lambda \in ]0,2].
 \end{equation}

We show that this system is locally  in time unconditionally well posed in $H^s \times H^{s+1}$ for $s> \frac{1}{2}$ uniformly with respect to the parameter $ \epsilon\ge 0 $ and $ \lambda>0 $. In particular we get the convergence in 
$ C([0,T]; H^s\times H^{s+1}) $  of the solutions to \eqref{2} towards the solutions to the Boussinesq equation \eqref{1} as the parameter $ \epsilon $ tends to $ 0 $. Then we prove that the analysis of Schonbek to establish the entropic estimate still work for \eqref{2} so that we can extend our solutions  for all positive times. Finally we prove that 
the low regularity entropic solutions  of the Boussinesq equation with $ u_0 \in H^1 $  and $ \xi $ in an Orlicz class can also be obtained as limits of regular solutions by regularizing the initial datas and using our main convergence results. We prove also the continuity of the flow map. All the previous results are obtained only under the non zero-depth condition $1+\zeta_0>0.$
 \subsection{Statement of the main results}
 \begin{definition}\label{def} Let $ s>1/2 $ and $T>0$ . We will say that $(\zeta,u)\in L^\infty(]0,T[;H^s\times H^{s+1}) $ is a solution to \eqref{2} associated with the initial datum $ (\zeta_0,u_0) \in H^s(\R)\times H^{s+1}(\R) $  if
  $ (\zeta,u) $ satisfies \eqref{2}   in the distributional sense, i.e. for any test function $ \psi\in C_c^\infty(]-T,T[\times \R) $,  it holds
  \begin{equation}\label{weak}
  \left\{ 
  \begin{array}{l}
  \int_0^\infty \int_{\R} \Bigl[(\psi_t +\psi_x+\epsilon g_\lambda(\psi))\zeta  +  \psi_x(\zeta u) \Bigr] \, dx \, dt +\int_{\R} \psi(0,\cdot) \zeta_0 \, dx =0\\
   \int_0^\infty \int_{\R} \Bigl[(\psi_t -\psi_{txx}+\psi_x)u   +  \psi_x  u^2/2 \Bigr] \, dx \, dt +\int_{\R} \psi(0,\cdot) u_0 \, dx =0
  \end{array}
  \right.
  \end{equation}
 \end{definition}
 \begin{remark} \label{rem2} Note that $ H^s(\R) $ is an algebra for $ s>1/2 $ and thus  $ \zeta u $ and $ u^2 $ are well-defined and belong to $L^\infty(]0,T[\, ;H^s(\R) $.   Moreover, $g_\lambda(\zeta)\in L^\infty(]0,T[\, ;H^{s-\lambda})$. 
  Therefore \eqref{weak} forces $ (\zeta_t,u_t)  \in L^\infty(]0,T[ \, ; H^{s-2}(\R)\times H^{s+1}) $  and  thus \eqref{2} is satisfied in $L^\infty(]0,T[ \,; H^{s-2}(\R)\times H^{s+1}) $. In particular, $ (\zeta,u)\in C([0,T] \, ; H^{s-2}(\R)\times H^{s+1})$ and \eqref{weak} forces 
 $ (\zeta(0),u(0))=(\zeta_0,u_0)$.
  \end{remark}
 \begin{definition} Let  $ s>1/2$. We will say that the Cauchy problem associated with \eqref{2} is unconditionally globally well-posed in $ H^s(\R )\times H^{s+1}(\R) $ if for any initial data $ (\zeta_0,u_0)\in H^s( {\R})\times H^{s+1}(\R)  $ there exists a solution
 $ (\zeta,u) \in C(\R_+\, ; H^s(\R)\times H^{s+1}(\R) ) $ to \eqref{2} emanating from $ (\zeta_0,u_0) $. Moreover, for  $T>0,$\; $(\zeta,u) $ is the unique solution to  \eqref{2} associated with $ (\zeta_0,u_0) $ that belongs to  $ L^\infty(]0,T[\, ;H^s(\R)\times H^{s+1}(\R) )$. Finally, for any $ T>0 $, the solution-map
  $(\zeta_0, u_0) \mapsto (\zeta,u) $ is continuous from  $ H^s(\R)\times H^{s+1}(\R) $    into $C([0,T]\,; H^s(\R)\times H^{s+1}(\R)) $.
 \end{definition}

 \begin{theorem}\label{maintheorem}
 For any $ \epsilon\ge 0 $, $ \lambda\in ]0,2]$ and any $ s>1/2 $, the Cauchy problem (\ref{2})  is unconditionally globally well-posed in $H^s(\R)\times H^{s+1}(\R)$.
 
 Moreover, denoting by $ (\zeta^{\epsilon,\lambda},u^{\epsilon,\lambda})$ the solution  to \eqref{2} emanating from 
 $(\zeta_0,u_0)\in H^s(\R) \times H^{s+1}(\R) $, for any $ T>0$ it holds 
 \begin{equation}\label{conv}
 (\zeta^{\epsilon,\lambda},u^{\epsilon,\lambda})\tendsto{\epsilon\to 0 }  (\zeta,
 u)
 \text{ in } C([0,T],H^{s}(\R)\times H^{s+1}(\R)) \; .
 \end{equation}
 where $ (\zeta,u) $ denotes  the solution to \eqref{1} emanating from  $(\zeta_0,u_0)$.
 \end{theorem}

 \section{Notations and preliminary}
 \subsection{Notations and function spaces.}
 In the following, $C$ denotes any nonnegative constant whose exact expression is of no importance. The notation $a\lesssim b$ means that $a\leq C_0 b$.\\
 We denote by $C(\lambda_1, \lambda_2,\dots)$ a nonnegative constant depending on the parameters
 $\lambda_1$, $\lambda_2$,\dots and whose dependence on the $\lambda_j$ is always assumed to be nondecreasing.\\
 Let $p$ be any constant with $1\leq p< \infty$ and denote $L^p=L^p(\R)$ the space of all Lebesgue-measurable functions
 $f$ with the standard norm $$\vert f \vert_{L^p}=\big(\int_{\R}\vert f(x)\vert^p dx\big)^{1/p}<\infty.$$ When $p=2$,
 we denote the norm $\vert\cdot\vert_{L^2}$ simply by $\vert\cdot\vert_2$. The real inner product of any functions $f_1$
 and $f_2$ in the Hilbert space $L^2(\R)$ is denoted by
\[(f_1,f_2)=\int_{\R}f_1(x)f_2(x) dx.
 \]
 The space $L^\infty=L^\infty(\R)$ consists of all essentially bounded, Lebesgue-measurable functions
 $f$ with the norm
\[
 \vert f\vert_{\infty}= \hbox{ess}\sup \vert f(x)\vert<\infty.
\]
 We denote by $W^{1,\infty}=W^{1,\infty}(\R)=\{f, \partial_x f\in L^{\infty}\}$ endowed with its canonical norm. For convenience, we denote the norm of $L^\infty(\R_+^*\times \R)$ by $\|\cdot\|_{L^\infty_{t,x}}.$\\

For any real constant $s\geq0$, $H^s=H^s(\R)$ denotes the Sobolev space of all tempered
 distributions $f$ with the norm $\vert f\vert_{H^s}=\vert \Lambda^s f\vert_2 < \infty$, where $\Lambda$
 is the pseudo-differential operator $\Lambda=(1-\partial_x^2)^{1/2}$.\\
 For any functions $u=u(t,x)$ and $v(t,x)$ 
 defined on $ [0,T)\times \R$ with $T>0$, we denote the inner product, the $L^p$-norm and especially
 the $L^2$-norm, as well as the Sobolev norm,
 with respect to the spatial variable $x$, by $(u,v)=(u(t,\cdot),v(t,\cdot))$, $\vert u \vert_{L^p}=\vert u(t,\cdot)\vert_{L^p}$,
 $\vert u \vert_{L^2}=\vert u(t,\cdot)\vert_{L^2}$ , and $ \vert u \vert_{H^s}=\vert u(t,\cdot)\vert_{H^s}$, respectively.\\
For $(X,\|\cdot\|_X)$ a Banach space, we denote as usually $L^p(]0,T[;X)$, $1\leq p\leq +\infty $,  the space of mesurable functions equipped by the norm:
 \[\big\Vert u\big\Vert_{L^p_TX} \ =\left( \displaystyle \int_0^T\big\Vert u(t,\cdot)\big\Vert_{X}^p \right) ^{1/p}  \hbox{ for }1\leq p<+\infty,\]
and
 \[\big\Vert u\big\Vert_{L^\infty_T X} \ = \ \hbox{ess}\sup_{t\in]0,T[}\Vert u(t,\cdot)\Vert_{X}\;\; \hbox{ for } p=+\infty  \ .\]
Finally, $C^k([0,T]; X)$  is the space of  $k$-times continuously differentiable functions from $[0,T]$ with value in $X,$ equipped with its standard norm \[\big\Vert u\big\Vert_{C^k([0,T];X)} \ = \ \max_{0\leq l\leq k} \sup_{t\in[0,T]}\vert u^{(l)}(t,\cdot)\vert_{X} \ .\] 
 Let $C^k(\R)$ denote the space of
 $k$-times continuously differentiable functions.  \\
 For any closed operator $T$ defined on a Banach space $X$ of functions, the commutator $[T,f]$ is defined
  by $[T,f]g=T(fg)-fT(g)$ with $f$, $g$ and $fg$ belonging to the domain of $T$. 
Throughout the paper, we fix a smooth cutoff function $\eta$ such that
\begin{displaymath}
\eta \in C_0^{\infty}(\mathbb R), \quad 0 \le \eta \le 1, \quad
\eta_{|_{[-1,1]}}=1 \quad \mbox{and} \quad  \mbox{supp}(\eta)
\subset [-2,2].
\end{displaymath}

We set  $ \phi(\xi):=\eta(\xi)-\eta(2\xi) $. For $l \in \mathbb \N\setminus\{0\}$, we define
\begin{equation}\label{defphi}
\phi_{2^l}(\xi):=\phi(2^{-l}\xi) .
\end{equation}
Any summations over  $N $ or $ K$  are presumed to be dyadic i.e. $ N$ and $ K$   range over numbers of the form $\{2^k : k\in\mathbb \Z\}$.  Then, we have that
{
\begin{displaymath}
\sum_{N>0 }\phi_N(\xi)=1\quad \forall \xi\in \R^* 
\; .
\end{displaymath}
Let us define the Littlewood-Paley multipliers by
\begin{displaymath}
P_Nu=\mathcal{F}^{-1}_x\big(\phi_N\mathcal{F}_xu\big), \quad
\tilde{P}_N u=(P_{2^{-1}N}+P_N+P_{2N})u \; 
\end{displaymath}
\begin{displaymath}
P_{\gtrsim N} :=\sum_{K\gtrsim N} P_K 
\; \text{and} \; 
P_{\ll N}:=\sum_{K \ll N} P_{K}
\end{displaymath}

\subsection{Some preliminary estimates.} The following product and commutator estimates will be used intensively throughout the paper. 
\begin{proposition}\label{propcom} Let  $ N>0 $ then 
\begin{equation}\label{cm1} |[P_N,P_{\ll N} f]g_x|_{L^2}  \lesssim |f_x|_{L^\infty}  |\tilde{P}_N g|_{L^2},
\end{equation}
\end{proposition}
We give a short proof of \eqref{cm1} in the appendix for sake of completness.

We will also need the two following product estimates in Sobolev spaces :
\begin{enumerate}
\item 
  For every $  p,r,t $ such that $ r+p-t>1/2 $ and $r,p\ge t  $,
\begin{equation}
\|fg\|_{H^t(\R)} \lesssim \|f \|_{H^p(\R)} \|g\|_{H^r(\R)}
\; . \label{prod2}
\end{equation}
\item For any $ s\ge 0 $ 
\begin{equation}
\|fg\|_{H^s(\R)} \lesssim \|f \|_{L^\infty} \|g\|_{H^s(\R)}+  \|f\|_{H^s(\R)}\|g \|_{L^\infty}
\; . \label{prod3}
\end{equation}
\end{enumerate}
Inequality (\ref{prod2}) is a standart Sobolev product estimate, the second one (\ref{prod3}) is the well known Moser product estimate (see for instance \cite{Taylor91} or  \cite{Lanlivre13}, and references therein.)
With \eqref{prod2}-\eqref{prod3}  in hand, it is straightforward (see Appendix) to prove the two following  frequency localized product estimates  given in proposition (\ref{propproduit}) that we will extensively use in the next section.
\begin{proposition}\label{propproduit}
For any $ N>0 $ and $ s>0 $  it holds 
\begin{equation}
N^s |P_N ( P_{\gtrsim N} f \, g_x)|_{L^2} \lesssim \delta_N \min \Bigl( |f|_{H^{s+1}} |g|_{L^\infty}, |f|_{H^s} |g_x|_{L^\infty}\Bigr) \label{proN1}
\end{equation}
 whereas for $ s>1/2 $ it holds 
\begin{equation}
N^{s-1} |P_N ( P_{\gtrsim N} f g_x)|_{L^2} \lesssim \delta_N |f|_{H^{s+1}} |g|_{H^{s-1}}\label{proN2}
\end{equation}
with $ |(\delta_{2^j})_{j\ge 0}|_{l^2}\le 1 $. 
\end{proposition}
We  also  need the following property of the regularizing operator defined in \eqref{op} (see Appendix).
\begin{proposition}\label{propgl} Let $f\in H^{\lambda/2+s}$, for $s\in \R_+$ . We have
 \begin{eqnarray}\label{gl1}
(g_\lambda[\Lambda^s f], \Lambda^s f)_{L^2}\geq |f_x|_{H^{\lambda/2-1+s}}.
\end{eqnarray}
\end{proposition}

\section{Local existence for the regularized system and energy estimates}

\subsection{Local well-posedness and estimates for a Bona-Smith's approximation}\label{31}
We fix $ \epsilon>0 $ in \eqref{2}. For $ \mu>0  $ we 
 consider the  Bona-Smith's type regularization  problem associated to (\ref{2})
\begin{eqnarray}\label{3l}
\left\{
\begin{array}{lcl}
\zeta_t- \mu \zeta_{txx} +u_x+(u\zeta)_x+ \epsilon g_\lambda(\zeta) &=&0,\\
u_t-u_{xxt}+\zeta_x+u u_x&=&0,\\
(\zeta,u)(0)&=(\zeta_0,u_0).& 
\end{array} \right.
\end{eqnarray}
Setting $ V=(\zeta,u)  $, \eqref{3l} can be rewritten as 
\begin{equation}
\frac{d}{dt} V = \Omega_\mu (V) 
\end{equation}
where 
$$
 \Omega_\mu (V)=\Bigl( (1-\mu \partial_x^2)^{-1}[-u_x -(u \zeta)_x - \epsilon g_\lambda(\zeta)],\;
(1-\partial_x^2)^{-1}[ -\zeta_x - u u_x ]\Bigr) 
$$
Since $ H^{s}(\R) $ is an algebra for $ s>1/2 $ and $ \lambda \le 2$, it is straightforward to check that $ \Omega_\mu $ is a  locally Lipschitz  mapping from $ (H^{s+1}(\R))^2 $ into itself for $ s>1/2$. Therefore by the Cauchy-Lipschitz theorem for ODE in Banach spaces we infer that 
\eqref{4} is locally well-posed in $(H^{s+1}(\R))^2 $, i.e. for any $ (\zeta_0,u_0)\in  (H^{s+1}(\R))^2 $ there exists $ 
 T_s=T_s(|\zeta_0|_{H^{s+1}}+|u_0|_{H^{s+1}})$ and a unique solution $(\zeta,u) \in C^1([0,T_s]; (H^{s+1})^2)$. Moreover, for any $R>0 $,  the mapping that to $ (\zeta_0,u_0) $ associates $(\zeta,u) $ is continuous from $ ( B(0,R)_{H^{s+1}})^2\subset (H^{s+1})^2 $ into 
 $ C([0,T_s(R)]; (H^{s+1})^2)$.

We start  by stating some energy estimate fundamental to prove our result.
For $ s\ge 0 $ and $ \mu\ge 0  $ we define $ E^s_\mu \; :\; (H^{s+1}(\R))^2 \to \R $ by 
\begin{equation}\label{ener}
E^s_\mu(\zeta,u) =|\zeta|_{H^s}^2+\mu |\zeta_x|_{H^s}^2+|u|_{H^{s+1}}^2
\end{equation}
In the sequel we denotes by $ (\delta_N)_{N\in 2^{\Z}} $  any  sequence of positive real numbers such that 
$$ \sum_{j\in \Z} \delta_{2^j}^2 \le 1\ .$$ 

\subsubsection{$H^s$ estimate.} 
Applying the operator $P_N$  to the equations in (\ref{3l}), multiplying respectively by $
\cro{N}^{2s}P_N \zeta$ and $\cro{N}^{2s} P_N   u$ the first and the second equation,  integrating with respect to $x$ and adding the resulting equations, we get
\begin{align}
\displaystyle \frac{\cro{N}^{2s} }{2}  \frac{d}{dt} E^0_\mu(P_N \zeta, P_N u) &+\epsilon \cro{N}^{2s} (g_\lambda[P_N \zeta],P_N\zeta)_{L^2}\nonumber \\
&=-\cro{N}^{2s} (P_N(\zeta u)_x,P_N \zeta)_{L^2}-2 \cro{N}^{2s}(P_N ( u u_x),P_N  u)_{L^2}.\label{E1s}
\end{align}  
We note that  Proposition \ref{propgl} yields 
$$ \cro{N}^{2s}(g_\lambda[P_N \zeta],P_N \zeta)_{L^2}\geq | P_N \zeta_x|^2_{H^{\lambda/2-1}} \ge 0 \; .$$
Integrating by parts and using  \eqref{cm1}  and \eqref{proN1}  we get
\begin{align*}
\cro{N}^{2s} |(P_N (u u_x) ,P_N u)_{L^2}|& =\cro{N}^{2s} |(P_N ((P_{\ll N}+P_{\gtrsim N})u u_x) ,P_N u)_{L^2}| \\
&=\cro{N}^{2s} \Bigl|-\frac{1}{2} (P_{\ll N} u_x P_N u, P_N u)_{L^2}  + \\
& \qquad ([P_N, P_{\ll N} u] u_x ,P_N u)_{L^2}+(P_N(P_{\gtrsim N} u u_x), P_N u) 
\Bigr| \\
&\lesssim \cro{N}^{2s} |u_x|_{L^\infty} |\tilde{P}_N u |^2_{L^2}+\delta_N N^s |P_N u|_{L^2}|u_x|_{L^\infty} |u|_{H^s}\; .
\end{align*} 
In the same way, integrating by parts and using \eqref{cm1} and \eqref{proN1} we obtain
\begin{align*}
\cro{N}^{2s}|(P_N(u \zeta_x),P_N  \zeta)_{L^2}|&=\cro{N}^{2s} \Bigl|-\frac{1}{2} (P_{\ll N} u_x P_N \zeta, P_N  \zeta)_{L^2}  +
([P_N, P_{\ll N} u]  \zeta_x ,P_N  \zeta)_{L^2} \\
& \qquad + (P_N(P_{\gtrsim N} u \,  \zeta_x), P_N  \zeta)_{L^2} \Bigr|\\
&\lesssim  \cro{N}^{2s}|u_x|_{L^\infty}  |\tilde{P}_N \zeta |^2_{L^2}+ \cro{N}^{s} \delta_N |\zeta|_{L^\infty}|u|_{H^{s{\color{red}+1}}} |P_N \zeta|_{L^2}\; .
\end{align*}
 While \eqref{prod3} leads to  
\begin{align*}
\cro{N}^{2s} |(P_N(u_x \zeta),P_N  \zeta)_{L^2}|& \lesssim \cro{N}^{s} \delta_N |u_x \zeta|_{H^s}  |P_N \zeta|_{L^2} \\
 & \lesssim \cro{N}^{s} \Bigl( \delta_N |u|_{H^{s+1}} |\zeta|_{L^\infty} +|u_x|_{L^\infty} |\zeta|_{H^s} \Bigr) |P_N \zeta|_{L^2} \
\end{align*}
Plugging  the three last inequalities in (\ref{E1s}), integrating on $ ]0,T[ $ and applying H\"older inequality in time one gets 
\begin{align*} 
 \displaystyle 
 |P_N \zeta|_{{L^\infty_T} H^s}^2 &+\mu  |P_N \zeta|_{{L^\infty_T} H^{s+1}}^2+ |P_N u|_{{L^\infty_T} H^{s+1}}^2+ \epsilon |P_N \zeta|_{L^2_T H^{s+\lambda/2-1}}^2
\lesssim  \cro{N}^{2s} E^0_\mu (P_N \zeta_0,P_N u_0)\nonumber \\
& + T^{1/2}\delta_N  (|u_x|_{L^\infty_{Tx}}+|\zeta|_{L^\infty_{Tx}})  (|u|_{L^\infty_T H^{s+1}}+|\zeta|_{L^\infty_T H^s})
  (|P_N \zeta|_{{L^2_T} H^s} + |P_N u|_{{L^2_T} H^{s+1}})
\end{align*}
Summing in $ N>0 $ and  applying Cauchy-Schwarz inequality  in $N$ on the last term to the above right-hand side member we eventually get 
\begin{align} 
 | \zeta|_{{L^\infty_T} H^s}^2 +&\mu  | \zeta|_{{L^\infty_T} H^{s+1}}^2+ | u|_{{L^\infty_T} H^{s+1}}^2+ \epsilon |\zeta|_{L^2_T H^{s+\lambda/2-1}}^2
\lesssim   E^s_\mu (\zeta_0,u_0)\nonumber \\
& + T^{1/2} (|u_x|_{L^\infty_{Tx}}+|\zeta|_{L^\infty_{Tx}})  (|u|_{L^\infty_T H^{s+1}}+|\zeta|_{L^\infty_T H^s})
  (| \zeta|_{{L^2_T} H^s} + |u|_{{L^2_T} H^{s+1}})\nonumber \\
  &\lesssim  E^s_\mu (\zeta_0,u_0,)
 + T  (|u_x|_{L^\infty_{Tx}}+|\zeta|_{L^\infty_{Tx}})  (|u|_{L^\infty_T H^{s+1}}^2+|\zeta|_{L^\infty_T H^s}^2)
\label{gf}
\end{align}
According to classical Sobolev inequalities, 
denoting by $ T^{\infty}_s $  the maximal time of existence in $ (H^{s+1}(\R))^2 $, The local well-posedness of \eqref{3l} in $ (H^{s+1}(\R))^2 $ together with \eqref{gf} ensure that 
for any $ s >1/2 $, $T^\infty_s= T^\infty_{\frac{1}{2}+} $. On the other hand, \eqref{gf} 
with $ s=\frac{1}{2}+ $ together with a  classical continuity argument ensure that $  T^\infty_{\frac{1}{2}+}\gtrsim 
[E^{\frac{1}{2}+}_\mu(\zeta_0,u_0)]^{-1/2} $ and that for any $s>1/2 $, 
 \begin{equation}
 \sup_{t\in [0,T_{\frac{1}{2}+,\mu}]} E^{s}_\mu(\zeta,u)(t) +\epsilon |\zeta_x|_{L^2_{T_{\frac{1}{2}+,\mu}} H^{s+\frac{\lambda}{2}-1}}^2\le 2 E^{s}_\mu(\zeta_0,u_0)
\label{defT0}
 \end{equation}
 with $ T_{\frac{1}{2}+,\mu}=T_{\frac{1}{2}+}(E^{\frac{1}{2}+}_\mu(\zeta_0,u_0))\sim [ 4 E^{\frac{1}{2}+}_\mu(\zeta_0,u_0)]^{-1/2} $.
\subsubsection{$H^{s-1}$ estimate for the difference of two solutions.}\label{313}
Let $(\zeta_i,u_i )$ be two solutions to \eqref{3l} with respectively $ \mu_1 $ and $ \mu_2 $, then setting $ \eta=\zeta_1-\zeta_2 $ and $ v=u_1-u_2 $ it holds 
\begin{eqnarray}\label{3ldif}
\left\{
\begin{array}{lcl}
\eta_t-\mu_1 \eta_{txx}+v_x+(u_1\eta)_x+ \epsilon g_\lambda(\eta) &=&(v\zeta_2)_x+(\mu_1-\mu_2) {\zeta_2}_{txx},\\
v_t+\eta_x+u_1 v_x-v_{xxt}&=& v {u_2}_x,\\
\end{array} \right.
\end{eqnarray}
 Applying the operator $P_N $  to the equations in (\ref{3}), multiplying respectively by $\cro{N}^{2(s-1)}P_N  \zeta$ and $\cro{N}^{2(s-1)} P_N  v$ the first and the second equation,  integrating with respect to $x$, adding the resulting equations and proceeding as above but with  \eqref{prod2} and \eqref{proN2} instead of \eqref{prod3} and  \eqref{proN1} we  get
 \begin{align}
\cro{N}^{2(s-1)}\frac{d}{dt}  E^{0}_{\mu_1}(P_N \eta,P_N v) &+2\epsilon \cro{N}^{2(s-1)} |P_N \eta_x|^2_{H^{\lambda/2-1}}  \lesssim
\delta_N \cro{N}^{s-1}  |{u_1}|_{H^{s+1}} \nonumber\\
&\times( |\eta|_{H^{s-1}}+ |v|_{H^{s-1}})   (|P_N \eta|_{L^2}+ |P_N v |_{L^2} ) \nonumber\\
&+\cro{N}^{2s-2}  |P_N v |_{L^2}|P_N(v {u_2}_x)|_{L^2}\nonumber \\
& +\cro{N}^{2s-2}  |P_N \eta|_{L^2} \Bigl(  |P_N (v\zeta_2)_x|_{L^2} 
 + |\mu_1-\mu_2| |P_N {\zeta_2}_{xxt}|_{L^2} \Bigr) . \label{difdif}
 \end{align}
Noticing that, since $ s>1/2 $ it holds 
$$
\cro{N}^{s-1} |P_N (v\zeta_2)_x|_{L^2}\lesssim |P_N(v \zeta_2))|_{H^s} \lesssim \delta_N |v \zeta_2|_{H^s} \lesssim \delta_N  |v|_{H^s}  |\zeta_2|_{H^s} 
$$
and that \eqref{prod2} leads to  
$$
\cro{N}^{s-1}  |P_N(v u_{2x}) |_{L^2}\le  \delta_N |v u_{2x} |_{H^{s-1}}\lesssim \delta_N |v|_{H^{s-1}} |u_2|_{H^{s+1}}\;.
 $$
Therefore integrating \eqref{difdif}  on $]0,T[$, we eventually get 
\begin{align*} 
 \displaystyle 
 |P_N \eta|_{{L^\infty_T} H^{s-1}}^2 +&\mu  |P_N \eta|_{{L^\infty_T} H^{s}}^2+ |P_N v|_{{L^\infty_T} H^{s}}^2+ \epsilon |P_N \zeta|_{L^2_T H^{s+\lambda/2-2}}^2\nonumber \\
& \lesssim  \cro{N}^{2{s-1}} E^0_\mu (P_N v(0), P_N \eta(0))+  |\mu_1-\mu_2|^2 |{\zeta_2}_t|_{L^2_T H^{s+1}}^2\nonumber \\
& \quad+ T^{1/2}\delta_N  (1+|{u_1}|_{L^\infty_T H^{s+1}}+|{u_2}|_{L^\infty_T H^{s+1}}+|\zeta_2|_{L^\infty_T H^s})\\
& \qquad \times (|v|_{L^\infty_T H^{s}}+|\eta|_{L^\infty_T H^{s-1}})
  (|P_N \eta|_{{L^2_T} H^{s-1}} + |P_N u|_{{L^2_T} H^{s}})
\end{align*}
Summing in $ N>0 $ and  applying Cauchy-Schwarz inequality  in $N$ on the last term to the above right-hand side member we obtain 
\begin{align} 
 | \eta|_{{L^\infty_T} H^{s-1}}^2 +&\mu  | \eta|_{{L^\infty_T} H^{s}}^2+ | v|_{{L^\infty_T} H^{s}}^2+ \epsilon |\eta|_{L^2_T H^{s+\lambda/2-2}}^2
\nonumber \\
& \lesssim   E^{s-1}_\mu (v(0), \eta(0))+ T |\mu_1-\mu_2|^2 |{\zeta_2}_t|_{L^\infty_T H^{s+1}}^2\nonumber \\
 & + T   (1+|{u_1}|_{L^\infty_T H^{s+1}}+|{u_2}|_{L^\infty_T H^{s+1}}+|\zeta_2|_{L^\infty_T H^s}) (|v|_{L^\infty_T H^{s}}^2+|\eta|_{L^\infty_T H^{s-1}}^2)
\label{dif}
\end{align}
\subsection{Local well-posedness of \eqref{2} uniformly in $\epsilon\in [0,1]  $ and $ \lambda\in ]0,2]$}
We will prove the local well-posedness of the regularized problem \eqref{2} using a standard compactness method. 
\begin{proposition}[Uniform in $ \epsilon $ and $ \lambda $ LWP]
 \label{essentiel}
Let $ s>1/2 $ and $(\zeta_0,u_0)\in H^s(\R) \times H^{s+1}(\R) $, then there exists $ T_0=T_0(|\zeta_0|_{H^{\frac{1}{2}+}}+|u_0|_{H^{\frac{3}{2}+}})$ such that for any $\epsilon\ge 0 $ and $ \lambda\in ]0,2]$ there exists a solution $(\zeta^{\epsilon,\lambda},u^{\epsilon,\lambda})$ of the Cauchy problem  (\ref{2}) in $ C([0,T_0]; H^s(\R)\times H^{s+1})$.   This is the unique solution to the IVP \eqref{2} that belongs to 
 $ L^\infty(]0,T_0[; H^s(\R)\times H^{s+1}) $.

Moreover,
$$
\sup_{\epsilon,\lambda} |(\zeta^{\epsilon,\lambda},u^{\epsilon,\lambda})
|_{L_{T_0}^{\infty}H^s\times H^{s+1}} \lesssim 
|(\zeta_0,u_0)|_{H^s\times H^{s+1}}$$
and  for any $\alpha>0 $, the solution map $S_{\epsilon,\lambda} :(\zeta_0,u_0) \longrightarrow (\zeta^{\epsilon,\lambda},u^{\epsilon,\lambda})$ is  continuous  
 from $B(0,\alpha)_{H^s\times H^{s+1} }  $ into $ C([0,T(\alpha)]; H^s(\R)\times H^{s+1}(\R)) $ uniformly in $\epsilon $ and $ \lambda$.
 
 Finally, let $ T^* $ be the maximal time of existence in $ H^s(\R) \times H^{s+1}(\R) $ of the solution  $(\zeta^{\epsilon,\lambda},u^{\epsilon,\lambda})$  emanating from $(\zeta_0,u_0)\in H^s(\R) \times H^{s+1}(\R) $. Then for any $ 0<T'<T^* $ it holds 
 \begin{equation}\label{exp}
  | \zeta|_{L^\infty_{T'} H^s}^2 + | u|_{L^\infty_{T'} H^{s+1}}^2  \lesssim \exp( C\; T' (|u_x|_{L^\infty_{T'x}}+|\zeta|_{L^\infty_{T'x}})) E^s_0 (\zeta_0,u_0)
 \end{equation}
 for some universal constant $ C>0 $.
 \end{proposition}
\begin{proof}
$\bullet $ {\it Unconditional uniqueness.} Let $ (\zeta_i,u_i)$, $i=1,2$ be two solution  of the IVP \eqref{2}  that belong to  $L^\infty(]0,T[; H^s(\R)\times H^{s+1})$  for some $ T>0 $.
Setting 
  $ \eta=\zeta_1-\zeta_2 $ and $ v=u_1-u_2 $, exactly the same calculations as in \ref{313} on the difference of two solutions  to \eqref{3l}  but with $ \mu_1=\mu_2=0 $   (note that all the calculus are justified since for any $ N $,  $ P_N u_i $ and $ P_N \zeta_i $ 
  belong to $ C^1([0,T]; H^\infty)$) lead  for $0<T'<T $ to 
\begin{align} 
|v|_{L^\infty_{T'} H^{s}}^2+ & |\eta|_{L^\infty_{T'} H^{s-1}}^2  \lesssim   E^{s-1}_0 (v(0), \eta(0))\nonumber \\
 & + T'   (1+|{u_1}|_{L^\infty_{T} H^{s+1}}+|{u_2}|_{L^\infty_{T} H^{s+1}}+|\zeta_2|_{L^\infty_{T} H^s}) (|v|_{L^\infty_{T'} H^{s}}^2+|\eta|_{L^\infty_{T'} H^{s-1}}^2)
\label{gfu}
\end{align}
that proves the uniqueness in this class by taking 
$$ 0<T'< (1+|{u_1}|_{L^\infty_{T} H^{s+1}}+|{u_2}|_{L^\infty_{T} H^{s+1}}+|\zeta_2|_{L^\infty_{T} H^s})^{-1}
$$
 and  repeating the argument a finite number of times.

\noindent
$\bullet $ {\it Existence.}
Let $(\zeta_0,u_0)\in H^s(\R) \times H^{s+1}(\R) $. We regularized the initial data by setting 
$\zeta_{0,n}=S_n \zeta $ and $u_{0,n}=S_n u_0 $ where $S_n $ is the Fourier multiplier by  $\chi_{[-n,n]} $. It is straightforward to check that 
 for $ n\ge 1$, $(\zeta_{0,n},u_{0,n})\in (H^\infty(\R))^2 $ with 
 \begin{equation}\label{gd}
    |u_{0,n}|_{H^{s+r}} \le n^r  |u_0|_{H^{s}}\quad\text{and}\quad |\zeta_{0,n}|_{H^{s+r
 }}\le n^r |\zeta_0|_{H^s}\quad \text{for any} \, r\ge 0   \; .
  \end{equation}
 Setting $ \mu= \mu_n=n^{-5} $, we thus obtain that  for any $ s>0 $ and any $ r\ge 0 $
  $$
E^{s+r}_{\mu_n}(\zeta_{0,n},u_{0,n})=   |\zeta_{0,n}|_{H^{s+r}}^2+n^{-5}  |\partial_x \zeta_{0,n}|_{H^{s+r}}^2+|u_{0,n} |_{H^{s+r+1}}^2
   \lesssim n^{2r}  E^s_0(\zeta_{0},u_{0})
   $$
   
  In particular setting, for $ s>1/2$,
 \begin{equation} \label{defTs}
  T_s \sim  (1+|{u_0}|_{ H^{s+1}}+|\zeta_0|_{ H^s})^{-1} \; , 
  \end{equation}
 we deduce from subsection \ref{31},  that  we can construct a sequence $ (\zeta_n,u_n)_{n\ge 1} \subset 
  C^1([0,T_{\frac{1}{2}+}]; (H^\infty(\R))^2) $
   such that for any $ n\ge 1 $, $ (\zeta_n,u_n)$ satisfies \eqref{3l} with $\mu= \mu_n=n^{-5} $.  
  Moreover, from \eqref{defT0}  and \eqref{gd} we infer  that for $ s>1/2 $ and $ r\ge 0 $
  \begin{align}
   \sup_{t\in [0,T_{\frac{1}{2}+}]} E^{s+r}_{\mu_n}(\zeta_n,u_n)(t) &  \le 2  E^{s+r}_{\mu_n}(\zeta_{0,n},u_{0,n})
   \nonumber \\
   & \lesssim n^{2r} E^s_{0}(\zeta_0,u_0) \; .  \label{313}
 \end{align}
 On the other hand from the first equation in \eqref{3l} we obtain that on $ [0,T_{\frac{1}{2}+}] $, 
\begin{align}
|\partial_t \zeta_n|_{H^{s+1}}&\le  |(1-\mu_n \partial_x^2)^{-1}\Bigl(u_{n,x}+(u_n\zeta_n)_x+ \epsilon g_\lambda(\zeta_n)  \Bigr)|_{H^{s+1}}\nonumber\\
& \le |u_{n,x}+(u_n \zeta_n)_x+ \epsilon g_\lambda(\zeta_n) |_{H^{s+1}}\nonumber\\
& \lesssim |u_n|_{H^{s+2}} (1+|\zeta_n|_{L^\infty} )+ |u_n|_{H^{s+1}} |\zeta_{n,x}|_{L^\infty}+ |\zeta_n|_{H^{s+1+\lambda}}
\nonumber \\
& \lesssim \sqrt{1+E_0^s(u_n,\zeta_n)}\sqrt{ E_0^{s+3} (u_n,\zeta_n)}\lesssim  n^{3}   (1+E_0^s(u_0,\zeta_0)) \label{ou}
\end{align}
For $ n_1\ge n_2 $ applying \eqref{dif} with $ (\zeta_i, u_i)=(\zeta_{n_i},u_{n_i}) $, $i=1,2$, using \eqref{313}-\eqref{ou} and that 
$  |\frac{1}{n_1^{5}}-\frac{1}{n_2^{5}}| \le \frac{1}{n_2^{5}}$  we thus obtain 
\begin{align}
  |\zeta_{n_1}-\zeta_{n_2}|_{L^\infty_{T_s} H^{s-1}}^2+& |u_{n_1}-u_{n_2}|_{L^\infty_{T_s} H^{s}}^2 
    \lesssim  E^{s-1}_{0}(\zeta_{0,n_1}-\zeta_{0,n_2}, u_{0,n_1}-u_{0,n_2})
+ \frac{1}{n^{4}_2} \label{315}
\end{align}
 that forces $ ((\zeta_n,u_n))_{n\ge 1}  $ to be a Cauchy sequence in $C([0,T_s]; H^{s-1}\times H^{s}) $.
  Since according to \eqref{defT0},  $((\zeta_n,u_n))_{n\ge 1}$ is bounded in  $C([0,T_s]; H^s\times H^{s+1}) $ with 
 $ (\zeta_{n})_{n\ge 1} $ bounded in $ L^2(]0,T_s[; H^{s+\frac{\lambda}{2}-1}) $ it follows that there exists $ (\zeta, u) \in 
 L^\infty ([0,T_s]; H^s\times H^{s+1}) $ with $\zeta\in   L^2(]0,T_s[; H^{s+\frac{\lambda}{2}-1}) $ such that 
 \begin{eqnarray}
 (\zeta_n, u_n) & \tendsto{n\to  +\infty} & (\zeta,u) \quad \text{in} \; C([0,T_s];H^{s'}\times H^{s'+1}) , \quad \forall 0<s'<s \\
 \zeta_n & \weaktendsto{n\to  +\infty} &\zeta \quad \text{in} \;  L^2(]0,T_s[; H^{s+\frac{\lambda}{2}-1})
 \end{eqnarray}
 In particular, $(\zeta,u) $ is a solution of the IVP \eqref{2}. \\
 \noindent
 $\bullet $ {\it Continuity in the strong norm} To prove the continuity of   $(\zeta,u) $ in $ H^s\times H^{s+1} $ we use Bona-Smith arguments to check that the sequence $ ((\zeta_n,u_n))_{n\ge 1}  $ is actually a Cauchy sequence in $C([0,T_s]; H^s\times H^{s+1}) $. 
  Let $ n_1\ge n_2 $ and set $ (\eta,v)=\zeta_{n_1}-\zeta_{n_2}, u_{n_1}-u_{n_2}) $, $ \mu_i=
  \mu_{n_i}=n_i^{-5} $. By the definition of $(\zeta_n,u_n) $  for any $ 0<r<s $
  \begin{equation}
  E^{s-r}_{\mu_{n_2}}(\eta(0),v(0))\le n_2^{-2r} E^{s}_{0}(\eta(0),v(0)) 
  \end{equation}
  Therefore, \eqref{315} together with \eqref{313}  and \eqref{defTs} ensure that 
  \begin{equation}\label{li} 
  |\eta|_{L^\infty_{T_s} H^{s-1}}^2+|v|_{L^\infty_{T_s} H^{s}}^2 \lesssim \frac{1}{n_2^2} E^{s} _{0}(\eta(0),v(0))+\frac{1}{n_2^4}
\le  \Bigl(\frac{1}{n_2}  \gamma(n_2)\Bigr)^2 \; .
  \end{equation}
  with $ \gamma(n) \to 0 $ as $ n\to +\infty $.  On the other hand, \eqref{313} ensures that for any $ r>0 $, 
   \begin{equation}
\sup_{t\in [0,T_{\frac{1}{2}+}]}  E^{s+r}_{\mu_{n_i}}(\zeta_{n_i}(t),u_{n_i}(t))\lesssim  n_i^{2r} E^{s} _{0}(\zeta_0,u_0)\; .\label{321}
  \end{equation}
  
  Now observing that $ (\eta,v)$ satisfies \eqref{3ldif} with $ (\zeta_i,u_i)=(\zeta_{n_i} ,u_{n_i}) $ and proceeding as in  \eqref{difdif} we eventually get 
  \begin{align}
N^{2s} \frac{d}{dt}E^0_{\mu_{n_1}}(P_N \eta,P_N v)&+2\epsilon  |P_N \eta_x|^2_{H^{s+\lambda/2-1}}
 \lesssim  |u_{n_1}|_{H^{\frac{3}{2}+}} ( |\eta|_{H^{s}}+ |v|_{H^{s}})  \nonumber\\
 & \times N^{s} (|P_N \eta|_{L^2}^2 + |P_N v|_{L^2}^2) +\delta_N N^s |P_N v|_{L^2} |u_{n_2}|_{H^{s+1}}|v|_{H^s} \nonumber \\
 & +\delta_N N^s   |P_N \eta|_{L^2} \Bigl(|v|_{H^{s}}
|\zeta_{n_2}|_{H^{s+1}}+ |v|_{H^{s+1}} |\zeta_{n_2}|_{H^s}  +n_2^{-5}|\partial_t \zeta_{n_2}|_{H^s}\Bigr)  \label{dif3}
\end{align}
But in view of \eqref{313} and  \eqref{li}
$$
 |\zeta_{n_2}|_{L^\infty_{T_s}H^{s+1}} |v|_{L^\infty_{T_s} H^s} \lesssim n_2 \frac{1}{n_2} \gamma(n_2)\tendsto{n_2\to +\infty} 0 
 $$
 and \eqref{ou} yields 
 $$
 \frac{1}{n_2^{4}} |\partial_t \zeta_{n_2}|_{L^\infty_{T_s}H^s} \lesssim \frac{1}{n_2} (1+E_0^s(u_0,\zeta_0)) \; .
 $$
Integrating in time and summing in $ N $, it thus follows that 
  \begin{align}
   | \eta|_{{L^\infty_{T_s}} H^{s}}^2 +\mu_{n_1}  | \eta|_{{L^\infty_{T_s}} H^{s+1}}^2+& | v|_{{L^\infty_{T_s}} H^{s+1}}^2+ 
2\epsilon  |\eta_x|^2_{L^2_{T_s} H^{s+\lambda/2-1}}\nonumber \\
&  \le E^s_{\mu_{n_1}}(\eta(0),v(0))+T_s \tilde{\gamma}(n_2)\nonumber \\
 & \quad+T_s ( 1+ |u_{n_1}|_{L^\infty_{T_s} H^{s+1}} +|u_{n_2}|_{L^\infty_{T_s} H^{s+1}}
 +|\zeta_{n_2}|_{L^\infty_{T_s} H^{s}} )  \nonumber \\
 & \qquad \times ( | \eta|_{{L^\infty_{T_s}} H^{s}}^2 +| v|_{{L^\infty_{T_s}} H^{s+1}}^2)
  \label{difs}
\end{align}
   $ (\zeta,u) \in C([0,T_s]; H^{s}\times H^{s+1}) $. 
Observe that 
$$
E^s_{\mu_{n_1}}(\zeta_{0,n_1}-\zeta_{0,n_2}, u_{0,n_1}-u_{0,n_2})\tendsto{n_1\to +\infty} E^s_0(\zeta_0-\zeta_{0,n_2},u_0-u_{0,n_2})$$
$$\hspace*{7.5cm}=E^s_{0}((1-S_{n_2})\zeta_0,(1-S_{n_2})u_0), 
$$
and thus letting $ n_1 \to +\infty $ in \eqref{difs}  we get
 \begin{align}
   \sup_{t\in [0,T_s]} E^{s}_{0}(\zeta-\zeta_n,u-u_n)(t) &  \lesssim   E^s_{0}((1-S_n)\zeta_0,(1-S_n)u_0) + \tilde{\gamma}(n)  \; ,\label{difh2s}
 \end{align}
with an implicit constant that is independent of $ \epsilon\ge 0 $ and $ \lambda\in ]0,2] $.\\
  $\bullet ${\it Continuity of the flow-map.} 
Let now  $ ((\zeta_{0}^k,u_{0}^k))_{k\ge 1} \subset
H^{s}(\R)\times H^{s+1}(\R) $ be such that
$
(\zeta_{0}^k,u_{0}^k) \rightarrow (\zeta_0,u_0)
$
in $ H^{s}(\R)\times H^{s+1}(\R) $.
We want to prove that the emanating solution $ (\zeta^k,u^k)$  to \eqref{2} tends to $ (\zeta,u)
$ in $ C([0,T_0];H^{s}\times H^{s+1}) $ uniformly in $ \epsilon $ and $ \lambda$. We set $ \zeta_{0,n}^k=S_n \zeta_0^k$ and $ u_{0,n}^k=S_n u_0^k$ and we call $(\zeta_n^k,u_n^k)\in C([0,T_s];H^{s}\times H^{s+1})$ the associated solution to \eqref{3l} with 
 $ \mu=\mu_n=n^{-5}$. 
By the triangle inequality, for $ k $ large enough, it holds  
$$
|u-u^k|_{L^\infty(]0,T_s[; H^{s})} \le
|u-u_n|_{L^\infty(]0,T_s[; H^{s})}+
|u_n-u^k_n|_{L^\infty(]0,T_s[; H^{s})}+
|u_n^k-u^k|_{L^\infty(]0,T_s[; H^{s})} \quad  .
$$
Using the estimate \eqref{difh2s} on the solution to \eqref{3l} we infer that
\begin{align}
  \sup_{t\in [0,T_s]} \Bigl( E^{s}_{0}(\zeta-\zeta_n,u-u_n)(t) &+ E^{s}_{0}(\zeta^k-\zeta^k_n,u^k-u^k_n)(t)\Bigr)\nonumber \\
 &  \lesssim E^s_{0}((1-S_{n})\zeta_0,(1-S_{n})u_0) \nonumber \\
 &\qquad +E^s_{0}((1-S_{n})\zeta_0^k,(1-S_{n})u_0^k)+
 \gamma(n) \label{kak-1}
\end{align}
and thus
\begin{equation}\label{kak1}
\lim_{n\to \infty} \sup_{k\in\N} \Bigl( |u-u_n|_{L^\infty_{T_s}
 H^{s}}+|u^k -u_n^k |_{L^\infty_{T_s} H^s}\Bigr) =0 \, .
\end{equation}
Therefore, it remains to prove that for any fixed $ n \in \N $, 
\begin{equation}
\lim_{k\to +\infty} |u^k -u_n^k |_{L^\infty_{T_s} H^s}=0\label{rem}
\end{equation}
For this we first   notice that \eqref{dif}  with $ \mu_1=\mu_2 $ ensures that
\begin{align}\label{326}
\|u_n-u_n^k\|_{L^\infty(]0,T_s[; H^{s})}^2 &\lesssim   E^{s-1}_{\mu_n} (\zeta_{0,n}-\zeta_{0,n}^k,u_{0,n}-u_{0,n}^k) \nonumber \\
& \lesssim E^{s-1}_{0} (\zeta_{0}-\zeta_{0}^k,u_{0}-u_{0}^k)\; .
\end{align}
and that \eqref{321} leads for $ r\ge 0  $ to 
\begin{equation}\label{327}
\sup_{t\in [0,T_{\frac{1}{2}+}]}  E^{s+r}_{\mu_{n}}(\zeta_{n}^k(t),u_{n}^k(t))\lesssim  n^{2r} E^{s} _{0}(\zeta_{0,n}^k,u_{0,n}^k)
\lesssim  n^{2r} (E^{s} _{0}(\zeta_0,u_0)+1)\; .
\end{equation}
Now, setting $ (\eta,v)=(\zeta_n-\zeta_n^k,u_n-u_n^k) $,  observing that $ (\eta,v)$ satisfies \eqref{3ldif} with $ (\zeta_1,u_1)=(\zeta_{n} ,u_{n}) $,
  $ (\zeta_2,u_2)=(\zeta_{n}^k,u_{n}^k) $ and $ 
\mu_1=\mu_2=n^{-5} $  and proceeding as in \eqref{dif3}  we 
 get 
    \begin{align}
\cro{N}^{2s} \frac{d}{dt}E^0_{\mu_{n}}(P_N \eta,P_N v)&+2\epsilon  |P_N \eta_x|^2_{H^{s+\lambda/2-1}}
 \lesssim  |u_{n}|_{H^{\frac{3}{2}+}} ( |\eta|_{H^{s}}+ |v|_{H^{s}}) \nonumber \\
 & \times \cro{N}^{s} 
 (|P_N \eta|_{L^2}^2 + |P_N v|_{L^2}^2) \nonumber \\
 & +\delta_N \cro{N}^s \Bigl(  |P_N v|_{L^2} |u_{n}^k|_{H^{s+1}}|v|_{H^s} 
 + |P_N \eta|_{L^2}|v|_{H^{s+1}} |\zeta_{n}^k|_{H^s}  \Bigr) \nonumber \\
 & +\delta_N \cro{N}^s  |P_N \eta|_{L^2} |v|_{H^{s}}
|\zeta_{n}^k|_{H^{s+1}}\; .  \label{dif33}
\end{align}
But \eqref{326}-\eqref{327} ensure that 
$$
 |v|_{H^{s}} |\zeta_{n}^k|_{H^{s+1}}\lesssim 
 n \, \Bigl[ (E^{s} _{0}(\zeta_0,u_0)+1)E^{s-1}_{0} (\zeta_{0}-\zeta_{0}^k,u_{0}-u_{0}^k)\Bigr] ^{1/2}\; .
$$ 
Therefore integrating in time and summing in $ N>0 $, it  follows that 
  \begin{align}
   | \eta|_{{L^\infty_{T_s}} H^{s}}^2 +& | v|_{{L^\infty_{T_s}} H^{s+1}}^2  \lesssim E^s_{0}(\eta(0),v(0))+T_s   n^2 (E^{s} _{0}(\zeta_0,u_0)+1) E^{s-1}_{0} (\zeta_{0}-\zeta_{0}^k,u_{0}-u_{0}^k)\nonumber \\
 & +T_s ( 1+ |u_{n}|_{L^\infty_{T_s} H^{s+1}} +|u_{n}^k|_{L^\infty_{T_s} H^{s+1}}
 +|\zeta_{n}^k|_{L^\infty_{T_s} H^{s}} ) 
  ( | \eta|_{{L^\infty_T} H^{s}}^2 +| v|_{{L^\infty_T} H^{s+1}}^2)
  \label{difs33}
\end{align}
which  ensures that 
$$
 | \eta|_{{L^\infty_{T_s}} H^{s}}^2 + | v|_{{L^\infty_{T_s}} H^{s+1}}^2 \lesssim 
  E^s_{0}(\eta(0),v(0))+T_s   n^2 (E^{s} _{0}(\zeta_0,u_0)+1) E^{s-1}_{0} (\zeta_{0}-\zeta_{0}^k,u_{0}-u_{0}^k)
  $$
  and proves \eqref{rem}. Note that this last estimate and \eqref{kak-1} are uniform in $ \epsilon $ and $ \lambda $.
 Combining \eqref{kak1} and \eqref{rem}, we thus obtain
 the continuity of the flow map in  $ C([0,T_s]; H^s\times H^{s+1})$ uniformly in $ \epsilon\ge 0 $ and $ \lambda\in ]0,2]$. Hence the IVP \eqref{2} is locally well-posed with a minimal time of existence $ T_s$ that satisfies \eqref{defTs}.

 Let now  $(\zeta_0,u_0) \in H^s\times H^{s+1} $ and $ T^*_s $ be the maximal time of existence in $ H^s\times H^{s+1} $ of  the emanating solution $(\zeta,u) $. Then proceeding exactly as to obtain \eqref{gf} in the preceding subsection we get for any $0<t_0<t_0+\Delta t<T' <T^*_s$, 
 \begin{align} 
 | \zeta|_{L^\infty(]t_0, t_0+\Delta t[; H^s)}^2 &+ | u|_{L^\infty(]t_0, t_0+\Delta t[; H^{s+1})}^2  \lesssim  E^s (\zeta(t_0),u(t_0))\nonumber \\
 &+ \Delta t  (|u_x|_{L^\infty_{T'x}}+|\zeta|_{L^\infty_{T'x}})  ( | \zeta|_{L^\infty(]t_0, t_0+\Delta t[; H^s)}^2 \nonumber \\
 &+ | u|_{L^\infty(]t_0, t_0+\Delta t[; H^{s+1})}^2)
\label{gfgf}
\end{align}
Therefore, for $ \Delta t \sim (|u_x|_{L^\infty_{T'x}}+|\zeta|_{L^\infty_{T'x}})^{-1} $, it holds 
$$
 | \zeta|_{L^\infty(]t_0, t_0+\Delta t[; H^s)}^2 + | u|_{L^\infty(]t_0, t_0+\Delta t[; H^{s+1})}^2  \lesssim E^s (\zeta(t_0),u(t_0))
 $$
 This proves \eqref{exp} by dividing $[0,T'] $ in small intervals of length $ \Delta t \sim (|u_x|_{L^\infty_{T'x}}+|\zeta|_{L^\infty_{T'x}})^{-1} $.
 
 Finally, \eqref{exp} and Sobolev embeddings  ensure that $ T_s^* =T_{\frac{1}{2}+}^* $ and thus the minimal time of existence in $ H^s\times H^{s+1} $ is bounded from below by $ T_{\frac{1}{2}+} $ that satisfies \eqref{defTs} with $ s=  \frac{1}{2}+$. 
  This completes the proof of  Proposition \ref{essentiel} with $T_0=T_{\frac{1}{2}+}$. 
 \end{proof}
 \subsection{Continuity of the flow-map with respect to the parameter $ \epsilon $.}
 It remains to prove the continuity of the flow-map with respect to the parameter $ \epsilon $  but this a direct consequence of the uniform in $ \epsilon $ LWP. Indeed, 
  let $ \lambda\in ]0,2] $ be fixed and let $(\zeta_0,u_0)\in H^s(\R) \times H^{s+1}(\R) $.  As in the preceding subsection, we set 
  $(\zeta_{0,n},u_{0,n})=(S_n \zeta_0,S_n u_0) $  and we denote by $(\zeta^{\epsilon}_n,u^{\epsilon}_n)\in C([0,T_{\frac{1}{2}+}]; H^s(\R)) $ the associated solution to \eqref{2}. For  $\epsilon\in \R_+ $ we have
\begin{align}
\norm{(\zeta^{\epsilon}-\zeta^{0},u^{\epsilon}-u^{0}}{L^\infty _{T_{\frac{1}{2}+}H^s}}
 & \le \norm{(\zeta^{\epsilon}-\zeta^{\epsilon}_n,u^{\epsilon}-u^{\epsilon}_n}{L^\infty _{T_{\frac{1}{2}+}}H^s} \nonumber \\
 & +  \norm{(\zeta^{\epsilon}_n-\zeta^{0}_n,u^{\epsilon}_n-u^{0}_n)}{L^\infty _{T_{\frac{1}{2}+}}H^s}\nonumber \\
 & +  \norm{(\zeta^{0}-\zeta^{0}_n,u^{0}-u^{0}_n)}{L^\infty _{T_{\frac{1}{2}+}}H^s}
.\label{triangle}
\end{align}
By the continuity of the flow-map uniformly in $ \epsilon\in \R_+$, the first and the third terms in the right-hand side can be made arbitrarily small by taking $ n $ large.
To estimate the second term, we set $(\eta,v)=(\zeta^{\epsilon}_n-\zeta^{0}_n,u^{\epsilon}_n-u^{0}_n)$ and we observe  that  $(\eta,v)$ satisfies 
\begin{eqnarray*}
\left\{
\begin{array}{lcl}
\eta_{t}+v_{x}+(u_n^{\epsilon} \eta)_x+ \epsilon g_\lambda(\eta_n) &=&
(v \zeta^{0}_n)_x-\epsilon g_\lambda(\zeta^{0}_n)
\\
v_t+\eta_x+u_n^\epsilon v_x-v_{xxt}&=& v \partial_x {u_n^0},\\
\end{array} \right.
\end{eqnarray*}
Proceeding as in the obtention of \eqref{difs33} (in particular, making use of \eqref{313}), we obtain for 
 $ 0<T\le T_{\frac{1}{2}+}$,
 \begin{align}
   | \eta|_{{L^\infty_{T}} H^{s}}^2 +& | v|_{{L^\infty_{T}} H^{s+1}}^2  \lesssim E^s_{0}(\eta(0),v(0))+ \epsilon T
    n^{4\lambda} E^{s} _{0}(\zeta_0,u_0)\nonumber \\
 & +T ( 1+n^2) \Bigl( 1+  E^s_{0}(\eta(0),v(0))\Bigr)
  ( | \eta|_{{L^\infty_T} H^{s}}^2 +| v|_{{L^\infty_T} H^{s+1}}^2)
  \label{difs33}
\end{align}
 Noticing that $ \eta(0)=v(0)=0 $  and proceeding as above we then get 
 $$
   | \eta|_{{L^\infty_{T}} H^{s}}^2 + | v|_{{L^\infty_{T}} H^{s+1}}^2
   \lesssim \exp \Bigl[ C\, T( 1+n^2) \Bigr] \epsilon T  n^{4\lambda} E^{s} _{0}(\zeta_0,u_0)
 $$
Taking $\epsilon$ sufficiently close to $0$ according to $n$, we see that the second term in the right-hand side of \eqref{triangle} can be made arbitrarily small.
Therefore, the convergence follows.

\section{A priori estimates and global existence of strong solutions}
In this section, we establish the global existence for any fixed $\epsilon\ge 0$ of (\ref{2}).
This completes the proof of Theorem \ref{maintheorem}.
To obtain the uniform estimates, we proceed as in \cite{Sch} by constructing a convex positive entropy for the associated hyperbolic system
\begin{eqnarray}\label{3}
\left\{
\begin{array}{lcl}
\zeta_t+(u+u\zeta)_x &=&0,\\
u_t+(\zeta+\frac{u^2}{2})_x &=&0.
\end{array} \right.
\end{eqnarray}
Let us we recall the notion of entropy for a hyperbolic system. Consider the system 
\begin{eqnarray}\label{H1}u_t+f(u)_x=0, \end{eqnarray}
where $u=u(t,x) \in \R^n$, $f: \R^n \longrightarrow \R^n$ a smooth function. We say that a pair of  functions 
$\eta, q\; : \R^n\to \R $  is an entropy-entropy flux pair if all smooth solutions of  (\ref{H1}) satisfy the additional conservation law
\begin{eqnarray}\label{H2} \eta(u)_t+q(u)_x=0, \end{eqnarray}
which can also be written
$$ \nabla \eta  u_t + \nabla q u_x=0.$$
On the other hand, multiplying (\ref{H1}) by $\nabla \eta$, we obtain
$$  \nabla \eta u_t+ \nabla\eta \nabla f  u_x =0.$$
This ensures that   the compatibility condition
\begin{eqnarray}\label{H3}\nabla \eta \nabla f=\nabla q, \end{eqnarray}
 forces any smooth solutions of  (\ref{H1}) to satisfy the additional conservation law \eqref{H2}. We define 
$$ w=1+\zeta,\, \sigma(w)= w\ln w,\, \sigma_L(w)=\sigma(1)+\sigma'(1)(w-1)=w-1 $$
and $$\sigma_0(w)=\sigma(w)-\sigma_L(w)=w \ln w -w+1.$$ 
Note that $\sigma_0$ is a convex function on $ ]0,+\infty[ $ and  enjoys the following property.
\begin{lemma} \label{rere}
Let  $ s>1/2 $ be fixed. The functional 
$$
\zeta \mapsto \int_\R\sigma_0(1+\zeta)dx  $$
is well-defined and continuous for the $L^\infty(\R)\cap L^2(\R) $ metric on the subset $\Theta $ of $ H^s(\R)  $ given by
$$
\Theta:=\{\zeta\in H^{s} (\R) , 1+\zeta>0 \; {\rm{on} }\; \R\} \; .
$$
Moreover, there exists  $ C>0 $ such that for all $ \zeta \in \Theta $, 
\begin{equation}\label{boundd}
0 \le  \int_\R\sigma_0(1+\zeta)dx \le C \int_\R \zeta^2 dx \; .
\end{equation}

\end{lemma}
\begin{proof} Let us fix $\zeta\in \Theta$. 
We first notice that since $s>1/2$, we have $\zeta\in C(\R) $ with $\zeta(x)\to 0 $ as $ |x|\to +\infty $ and thus $1+\zeta $ has got a minimum value  $ \alpha_0 \in ]0,1] $ on $ \R $. Therefore, for  $ \zeta'\in \Theta $ such that $
|\zeta-\zeta'|_{L^\infty} \le \alpha_0/2  $ it holds 
 \begin{equation}
1+\zeta'\ge \min_{\R} (1+\zeta)-\alpha_0/2=\alpha_0/2>0 \; .\label{tu}
\end{equation}
 Now, clearly $  \sigma_0'(1+z)=\ln(1+z) $ and thus $ 0\le \sigma_0'(1+z) \le z $ for $ z\ge 0 $. On the other hand, by the mean-value theorem, 
 for $ z\in [\alpha_0/2-1, 0] $ it holds $ |\ln(1+z)| \le \frac{2}{\alpha_0} |z|$.
  Gathering these two estimates and using again the mean value theorem we thus infer that 
  $$
  |\sigma_0(1+\zeta)-\sigma_0(1+\zeta')| \le \frac{2}{\alpha_0}\max(|\zeta|,|\zeta'|)|\zeta-\zeta'| \; 
 $$
 that yields 
  \begin{equation}\label{difff}
 \Bigl| \int_{\R} \sigma_0(1+\zeta)- \int_{\R} \sigma_0(1+\zeta')\Bigr| \le \frac{2}{\alpha_0} 
 (|\zeta|_{L^2} +|\zeta'|_{L^2}) |\zeta-\zeta'|_{L^2}\; \;.
 \end{equation}
 Taking $ \zeta'\equiv 0 $ we obtain that $ \int_\R\sigma_0(1+\zeta)dx$ is well-defined on $ \Theta$ and 
  the continuity result follows as well  from \eqref{difff}.
  
 Finally, we notice that 
    as $\sigma_0(1+x) \sim x \ln x $ at $ +\infty $ and 
$\sigma_0(1+x) \sim x^2 $ near the origin,  
there exists $ M\ge 1  $ and $c^M_1, \, c^M_2 >0 $ such that  
\begin{eqnarray}
(c_1^M)^{-1} x^2\ge 
\sigma_0(1+x) \ge c^M_1 x^2 \;  &  \text{for} & \; -1<x<M \nonumber \\
\quad \text{and}\quad (c^{M}_2)^{-1} x^2 \ge \sigma_0(1+x) \ge c^{M}_2 x \ln x \ge x   & \text{ for }  &x\ge M \; .\label{estA} 
\end{eqnarray}
This clearly leads to \eqref{boundd}.
\end{proof}
We introduce the Orlicz class associated to the function $\sigma_0(1+\cdot)$ $$\Lambda_{\sigma_0}:= \left\{\zeta\; \mbox{measurable}\;/ \displaystyle \int_{\R}\sigma_0(1+\zeta(x))\,dx<+\infty \right\}, $$ 
with the notation $|\zeta|_{\Lambda_{\sigma_0}}:= \displaystyle{\int_{\R}  \sigma_0(1+\zeta(x))\,dx}.$\\
Now, we shall establish a uniform crucial entropic estimate for our solution.
\begin{proposition}\label{thm1} 
Let $(\zeta_0,u_0)\in  H^s\times H^{s+1}$, for $s>1/2$, and such that $1+\zeta_0>0$. Then, the solution  
$(\zeta,u)\in C([0,T_0]; H^s\times H^{s+1}) $ to \eqref{2}, constructed in Proposition \ref{essentiel}, satisfies 
 $ 1+\zeta(t,x)>0 $ a.e. on $ [0,T_0]\times\R $ with $\zeta\in L^\infty(]0,T_0[; \Lambda_{\sigma_0}) $  and it holds 

\begin{align}
 \frac{1}{2} |u(t)|_{H^1}^2 + \int_{-\infty}^{+\infty}  \sigma_0(1+\zeta(t,x))dx\
  \le  \frac{1}{2} |u_0|_{H^1}^2+ \int_{-\infty}^{+\infty}  \sigma_0(1+\zeta_0(x))dx, \; \forall t\in [0,T_0] .\label{inq1}
 \end{align}
\end{proposition}
\begin{proof} 
We first assume that  $(\zeta_0,u_0) \in  (H^\infty(\R)\cap W^{2,1}(\R)) \times H^\infty(\R)$. 
According to Proposition \ref{essentiel}, \eqref{2} has got a unique solution $(\zeta,u)\in C([0,T_0]; H^\infty\times H^{\infty}) $ emanating from  $(\zeta_0,u_0)$, where $ T_0 $ only depends on $ |\zeta_0|_{H^{\frac{1}{2}+} }+
  |u_0|_{H^{\frac{3}{2}+}} $.
Then we observe that for $\theta=0,1,2 $,  $\Lambda^\theta \zeta $ verifies  the following integral representation on $ [0,T_0] $
\begin{eqnarray}\label{I1}
\begin{array}{lcl}
\Lambda^\theta\zeta(t,x)&=& \displaystyle \int_{-\infty}^{+\infty} \Lambda^\theta\zeta_0(z)K_\lambda(t,x-z)dz\\  \\&+ & \displaystyle \int_0^t \int_{-\infty}^{+\infty} 
\Lambda^\theta\partial_x \Bigr(u(s,z)+\zeta(s,z)u(s,z)\Bigl) K_\lambda(t-s,x-z))dz ds,
\end{array}
\end{eqnarray}
where $K_\lambda(t,x)= \mathcal{F}^{-1}(e^{-\epsilon t |\cdot|^\lambda})(x)$ is the kernel associated to $g_\lambda$ that satisfies (see \cite{DGV03}):
\begin{eqnarray}\label{I2} 
\|K_\lambda(t,\cdot)\|_{L^1(\R)}=1, \; \mbox{ and }\|\partial_x K_\lambda(t,\cdot)\|_{L^1(\R)}=c_1 (t\epsilon)^{-1/\lambda}, \forall t\in \R.
\end{eqnarray} 
and for $  (t,x) \in ]0;+\infty[ \times \R$
\begin{eqnarray}\label{Ig2} 
K_\lambda(t,x)=\displaystyle \frac{1}{t^{1/\lambda}}K_\lambda(1,\frac{x}{t^{1/\lambda}}),\, \, |K_\lambda(t,x)|\leq \displaystyle \frac{C}{t^{1/\lambda}(1+t^{-2/\lambda}|x|^2)}.
\end{eqnarray}

In particular, we obtain that 
\begin{equation}\label{W11}
\zeta\in L^\infty(]0,T_0[;W^{2,1}(\R)) \; .
\end{equation}

Now we make a change of unknown  by setting $w=1+\zeta$, the system (\ref{2}) becomes
\begin{eqnarray}\label{4}
\left\{
\begin{array}{lcl}
w_t+(uw)_x &=&-\epsilon g_\lambda(w),\\
u_t+(w+u^2/2)_x&=&u_{xxt},
\end{array} \right.
\end{eqnarray}
with initial data $w(0,x)=w_0(x)=1+\zeta_0(x)$ and $u(0,x)=u_0(x).$ The associated hyperbolic system becomes 
\begin{eqnarray}\label{5}
\left\{
\begin{array}{lcl}
w_t+(uw)_x &=&0,\\
u_t+(w+u^2/2)_x&=& 0.
\end{array} \right.
\end{eqnarray}
As in (\ref{H1}), let $f:\R^2 \longrightarrow \R^2$ be defined by 
$$ f(w,u)=(wu, w +u^2/2).$$
Let $\eta(\cdot,\cdot) $ and $q(\cdot,\cdot)$  be a pair of functions satisfying the compatibility condition (\ref{H3}). Then,  setting $ V=(w,u)^T $, the solution of (\ref{4}) satisfies  
\begin{eqnarray}\label{H5} 
\eta(w,u)_t+ q(w,u)_x &= & \nabla \eta(V) V_t +\nabla q(V) V_x\nonumber \\
& = & \nabla \eta\Bigl( V_t + (f(V))_x\Bigr)\nonumber\\
& = &\nabla \eta (V) \Bigl(\epsilon w_{xx}, u_{txx} \Bigr)^T\nonumber \\
& = & \epsilon \eta_w g_\lambda(w)+\eta_u u_{xxt}. \end{eqnarray}
Let $\eta $ be the function of the form
$$ \eta(w,u)=u^2/2+\alpha(w),$$
for some function $\alpha$. Thus, (\ref{H5}) becomes
\begin{eqnarray}\label{H6}\begin{array}{lcl} \eta(w,u)_t+ q(w,u)_x&=&\epsilon \alpha'(w) g_\lambda(w)+u u_{xxt}\\
&=&\epsilon \alpha'(w) g_\lambda(w)-(u_x^2/2)_t+ (uu_{xt})_x . \end{array}\end{eqnarray}
In order to get an a priori estimate on solutions to \eqref{4}, we have to choose $\alpha$ to be a convex function (see  Lemma \ref{lem1} below) and  $ \eta$ to be a convex  and positive function. Since $ u\mapsto u^2/2 $ is convex and positive, it actually suffices to ask $ \alpha $ to be also convex and positive.   At this stage, it is worth noticing that  Proposition \ref{minmax} in the Appendix ensures that $1+\zeta(t)
\ge \displaystyle \min_{\R} (1+\zeta_0) $ on $\R $ for any $ t\in [0,T_0]$. We set $\alpha(w)=\sigma(w) =w\ln w$ that is a convex function on 
 $ ]0,+\infty[ $. 
It is straightforward to check that with this $ \alpha $,  $\eta $ satisfies the compatibility condition \eqref{H3} with the entropy flux given by 
$$
q(w,u)=\alpha(w) u + u w + u^3/3 \; .
$$
Thus we have found an entropy which is convex but not positive. To obtain a positive convex entropy $\eta $, it suffices to substract from $\alpha $ its linear part at $ 1 $ that leads to 
$$
\tilde{\alpha}(w)= \sigma(w) - \sigma'(1) (w-1)= w \ln w +w-1=\sigma_0(w) \; ,
$$
so that the  entropy function becomes
\begin{equation}\label{defen}
\tilde{\eta}(w,u)=u^2/2+ w\ln w+w-1\; .
\end{equation}
Note that, in order for \eqref{H3} to hold, the entropy flux function $ q$ has to be modified consequently and becomes
$$ \tilde{q}(w,u)=q(w,u)-q(1,0)-\sigma'(1)[f(w,u)-f(1,0)]\;,$$
where we choose the constant so that $ \tilde{q}(1,0)=0 $.
Now, by omitting the tilde, the new $\eta$ and $q$ satisfy the equation (\ref{H6}) which will be the starting point of our calculations. As in \cite{Sch}, we consider the space time square
$$ R= \left\{ (s,x)\in \R^2/ 0\leq s\leq t , \, -N\leq x\leq N \right\}, \, \mbox{ for } N  >0.$$
Then  integrating equation (\ref{H6}) over $R$ and using the divergence theorem, we get
\begin{eqnarray*}\begin{array}{lcl} 
& \displaystyle
\int_{-N}^{N}& (\eta(w,u)(t,x)-\eta(w,u)(0,x))dx +\int_0^t (q(w,u)(s,N)-q(w,u)(s,N))ds \\ \\
&=& \displaystyle-\epsilon\int_0^t \int_{-N}^N \alpha'(w) g_\lambda(w)dx ds -\int_{-N}^{N}(u_x^2/2(t,x)- u_x^2/2(0,x))dx \\ \\ \displaystyle & &+\int_0^t (uu_{xt}(s,N)-uu_{xt}(s,-N))  ds.  \\ \\ \end{array}\end{eqnarray*}

Since for any fixed $t\in [0,T_0]$, $\zeta(t)\in H^\infty(\R) $ and  $w(t)=1+\zeta(t) >0$ on $\R $, we deduce that there exists $w_{\min}(t), \,w_{\max}(t)>0$ such that $ w(t)\in[w_{\min}(t),w_{\max}(t)]$. Clearly the mapping $\sigma_0
\, :\, w\mapsto w \ln w +1-w  $ belongs to $C^2([w_{\min}(t),w_{\max}(t)])$ with $ \sigma_0(1)=0 $ and  using  a regular convex and positive extension of $ \sigma_0 $ we can assume that $\sigma_0 $ is a  $C^2(\R)$ convex positive function that belongs to $ W^{2,\infty}(\R)$. Therefore \eqref{W11} ensures that $ \sigma_0(w(t)) \in W^{2,1}(\R) $ for any $ t\in [0,T_0] $.
  Letting $N$ go to $+\infty$ in the above equality, we can thus use  the Lebesgue dominated convergence theorem to get 
\begin{eqnarray}\label{H77}
\begin{array}{lcl}
& &\displaystyle \int _{-\infty}^{+\infty}\eta(w,u)(t,x)dx +\int _{-\infty}^{+\infty}\frac{u_x^2}{2}(t,x) dx\\  \\& &=\displaystyle 
\int _{-\infty}^{+\infty}\eta(w,u)(0,x)+\int _{-\infty}^{+\infty}\frac{u_x^2}{2}(0,x) dx-\epsilon\int_0^t \int_{\R} \sigma_0'(w) g_\lambda(w)dx ds. 
\end{array}
\end{eqnarray}
We need the  following Lemma to conclude. 
\begin{lemma}\label{lem1}
Let $\lambda \in ]0,2[,\varphi \in C_b^2(\R^n) $ and $\alpha \in C^2(\R)$ be a convex function. Then, we have
$$g_\lambda(\alpha(\varphi))\leq \alpha'(\varphi)g_\lambda(\varphi)$$
 \end{lemma}
For the proof of this lemma, see  \cite{DI}.

Now, let us treat the last term in (\ref{H77}). For each $ t\in [0,T_0] $ Lemma \ref{lem1} yields 
$$
g_\lambda(\sigma_0(w(t)))\leq \sigma_0'(w(t))g_\lambda(w(t))\; .
$$
and  (\ref{H77}) leads to 
\begin{eqnarray}\label{H8}
\begin{array}{lcl}
& &\displaystyle \int _{-\infty}^{+\infty}\eta(w,u)(t,x)dx +\int _{-\infty}^{+\infty}\frac{u_x^2}{2}(t,x) dx\\  \\& &\leq\displaystyle 
\int _{-\infty}^{+\infty}\eta(w,u)(0,x)+\int _{-\infty}^{+\infty}\frac{u_x^2}{2}(0,x) dx-\epsilon\int_0^t  \int_{\R}  g_\lambda(\sigma_0(w(s,x))dx ds.
 \end{array}
  \end{eqnarray}
 Now since $ \sigma_0(w(t))\in W^{2,1}$ for each $t\in [0,T_0] $, we get that $g_\lambda(\sigma_0(w(t,\cdot))\in L^1(\R)$. Indeed, this result is direct for $ \lambda=2 $ and for $ 0<\lambda<2 $ it suffices to notice that  
  $$
  P_N g_\lambda(f)={\mathcal F}_x^{-1} (|\cdot|^{\lambda} \phi_N(\cdot) ) \ast f ={\mathcal F}_x^{-1} \Bigl(
  \xi \mapsto \frac{|\xi|^{\lambda}}{\xi^2}  \phi_N(\xi) \Bigr) \ast \partial_x^2 f \; , \quad \forall f\in W^{2,1}(\R) ,
  $$
   where $ \phi_N $ is defined in \eqref{defphi} . It follows that 
   $$
   |P_N g_\lambda(f)|_{L^1} \lesssim \min(N^\lambda |f|_{L^1}, N^{\lambda-2} |\partial_x^2 f|_{L^1})\quad \forall N>0, 
   $$
  and thus 
  $$
  |g_\lambda(f)|_{L^1} \lesssim |\sum_{N>0} |P_N g_\lambda(f)|_{L^1}\lesssim 
 \sum_{0<N<1} N^{\lambda} |f|_{L^1}+  \sum_{N>1} N^{\lambda-2}  |\partial_x^2 f|_{L^1}\lesssim |f|_{W^{2,1}}\; .
 $$
 that proves the desired result. 
 Finally, since
 $$
   \mathcal{F}\Bigl( g_\lambda(\sigma_0(w(t,\cdot))\Bigr)(0)=0
   $$
 this ensures that   
 $$
 \int_{\R}  g_\lambda(\sigma_0(w(s,x))dx=0 , \quad \forall t\in [0,T_0] \; .
 $$ 
  This proves \eqref{inq1} for $(\zeta_0,u_0)\in (H^\infty(\R)\cap W^{2,1}(\R)) \times H^\infty(\R)$. The result for 
   $(\zeta_0,u_0)\in (H^s(\R)\times H^{s+1}(\R)) $ follows by  using the continuity of the flow-map together with { Lemma } \ref{rere}. Note in particular that the continuity in $ C([0,T]; H^{s}) $ of the flow-map associated with $ \zeta $  and Proposition \ref{minmax} (see the appendix) ensure that $ 1+\zeta(t,x) \ge 
 \displaystyle  \min_{x\in\R}(1+\zeta_0(x))>0 $ a.e. on $[0,T_0]\times \R $.
\end{proof}
Next, we state the global well-posedness  result.  
\begin{proposition}\label{prop1} Let  $(\epsilon, \lambda) \in \R_+\times ]0,2] $  and 
 let $(\zeta_0,u_0)\in H^s(\R)\times H^{s+1}(\R)$,  $s>1/2$, such that 
 $1+\zeta_0 >0$ . Then the unique solution $(\zeta,u) $ to  \eqref{2} constructed in Proposition \ref{essentiel} 
  can be extended for all positive times and thus belongs to $ C(\R_+;H^{s}\times H^{s+1})$. Moreover, for any $ T>0 $ there exists a constants $ C_{T,s}>0 $ only depending on $|\zeta_0|_{H^s}$ and $|u_0|_{H^{s+1}}$  such that 
  \begin{equation}
  |\zeta|_{L^\infty(]0,T[;H^s)}+  |u|_{L^\infty(]0,T[;H^{s+1})}\le C_{T,s}
  \end{equation}
 and the flow-map $S_{\epsilon,\lambda} :(\zeta_0,u_0) \longrightarrow (\zeta^{\epsilon,\lambda},u^{\epsilon,\lambda})$ is  continuous  
 from $H^s\times H^{s+1}  $ into $ C([0,T]; H^s(\R)\times H^{s+1}(\R)) $ uniformly in $\epsilon $ and $ \lambda$.
  \end{proposition}
  \begin{proof}
  According to \eqref{exp}  and the local well-posedness result, it suffices to proves that for any $ T>0 $ there exists $ c_T>0 $ only depending on $ T $, $|u_0|_{H^\frac{3}{2}+}$ and  $|\zeta|_{H^{\frac{1}{2}+}}$ , such that 
  if the solution $(\zeta,u) $ to \eqref{2} belongs to $ C([0,T[;H^s\times H^{s+1})$ then 
 \begin{equation}
 |\zeta|_{L^\infty(]0,T[\times \R)}+|u_x|_{L^\infty(]0,T[\times \R)}\le c_T \; .\label{estest}
 \end{equation}
   We mainly follow the proof of Theorem $1.2$ in \cite{Ami}.  Let $N$ be a positive odd integer, we start by deriving an estimate on 
  $ \sup_{t\in [0,T[} |\zeta(t)|_{L^N} $.
   For this we multiply the first of (\ref{2}) by $\zeta^N$ and integrate with respect to $x$, to get
$$\frac{1}{N+1}\frac{d}{dt}\int_{\R} \zeta^{N+1}+\epsilon \int_{\R} g_\lambda(\zeta)\zeta^N=-\int_{\R}  \zeta^N u_x - \frac{N}{N+1}\int_{\R}  \zeta^{N+1}u_x.
$$
To treat the  term $\epsilon \int g_\lambda(\zeta)\zeta^N$, we use the property of operator $g_\lambda$ in Lemma \ref{lem1} to prove that it is non negative. Note that the convexe function taking here $\alpha(x)=x^{N+1}$. 
Therefore integrating the above identity on $(0,t) $, using that $ N+1$ is an even integer, we get 
\begin{equation} \frac{1}{N+1}|\zeta(t) |_{L^{N+1}}^{N+1}\le  \frac{1}{N+1}|\zeta_0 |_{L^{N+1}}^{N+1}
 -\int_0^t \int_{\R}  \zeta^N u_x - \frac{N}{N+1}\int_0^t \int_{\R}  \zeta^{N+1}u_x.\label{LM}
\end{equation}
Now, we make use of the fact that for any  $ f\in L^2(\R) $ it holds 
 $(1-\partial_x^2)^{-1} f = \frac{1}{2} e^{-|\cdot|} \ast f $ and $ \partial_x^2(1-\partial_x^2)^{-1} f=-f+(1-\partial_x^2)^{-1} f $. Differentiating the second equation of \eqref{2} with respect to $ x$ we thus obtain 
\begin{eqnarray}
u_{tx} &= & \zeta-  \frac{1}{2} \int_{\mathbb{R}} e^{-|\cdot-z|}\zeta\, dz     +\frac{u^2}{2} - \frac{1}{4} \int_{\mathbb{R}} e^{-|\cdot-z|}u^2 (z)dz \nonumber \\
& =& \zeta+f_1+f_2+f_3 \; .\label{ux}
\end{eqnarray}
We would like to estimate the $ L^\infty $ and the $ L^2$-norms of the $f_i $.
The terms with $u $  in the above right-hand side can be easily estimate in the following way 
$$
\Bigl|f_2+f_3 \Bigr|_{L^\infty}
 \le |u|_{L^\infty}^2(\frac{1}{2} + \frac{1}{4} |e^{-|\cdot|}|_{L^1}) \le |u|_{H^1}^2 \; 
$$
and 
$$
\Bigl|f_2+f_3 \Bigr|_{L^2} \le |u|_{L^4}^2+|e^{-|\cdot|} \ast u^2|_{L^2} 
\lesssim  |u|_{L^4}^2\lesssim  |u|_{H^1}^2 \; .
$$
To estimate $f_1$ we will make use of \eqref{estA}.
Denoting by $A(t) $ the measurable set of $ \R $ defined by 
$$
A(t)=\{z\in \R \, ,/\, \zeta(t,z) \ge M\}\; ,
$$
Young's convolution estimates lead to 
\begin{eqnarray*}
 |e^{-|\cdot|}\ast \zeta  |_{L^\infty} & \le &\Bigl|  e^{-|\cdot|}\ast (\zeta \chi_{A^\complement})\Bigr|_{L^\infty} +
 \Bigl|  e^{-|\cdot|}\ast (\zeta \chi_{A})\Bigr|_{L^\infty}\\
& \le  &  |e^{-|\cdot|}|_{L^1} |\zeta \chi_{A^\complement}|_{L^\infty} + |e^{-|\cdot|}|_{L^\infty} |\zeta \chi_{ A}|_{L^1} \\
 & \le & 2M + |\zeta|_{\Lambda_{\sigma_0}}
\end{eqnarray*}
 and 
\begin{eqnarray}\label{esta1}
\begin{array}{lcl}
 |e^{-|\cdot|}\ast \zeta  |_{2} & \le  &  |e^{-|\cdot|}|_{L^1} |\zeta \chi_{A^\complement}|_{L^2} + |e^{-|\cdot|}|_{L^2} |\zeta \chi_{ A}|_{L^1} \\
 & \le & \frac{1}{C^M_1} |\zeta|_{\Lambda_{\sigma_0}}+  |\zeta|_{\Lambda_{\sigma_0}}\lesssim |\zeta|_{\Lambda_{\sigma_0}}.\end{array}
\end{eqnarray}
Integrating \eqref{ux} on $ [0,t] $ we get 
\begin{equation}\label{uh}
u_x(t) = u_{0,x} + \int_0^t \zeta(s) \, ds + F 
\end{equation}
where, according to  the above estimates and   Proposition \ref{thm1}, 
$$
|F(t)|_{L^\infty}+|F(t)|_{L^2}  \lesssim t \Bigl( 1+ |u_0|_{H^{1}}^2+ |\zeta_0|_{\Lambda_{\sigma_0}}\Bigr),
 \quad \forall t\in [0,T[  \; .
$$
Making use of Holder's inequality, this enables to bound  the first  term of the right-hand side member to \eqref{LM} in the following way :
\begin{align}
\Bigl| -\int_0^t & \int_{\R} \zeta^N(s) u_x(s) \, ds \Bigr| = \Bigl|  -\int_{\R}\int_0^t  \zeta^N (u_{0,x}+F) -  \int_0^t \int_{\R} \zeta^N\int_0^s \zeta(\tau) \, d\tau \Bigr| \nonumber \\
& \le     \Bigl( |u_{0,x}|_{L^{N+1}}+ |F|_{L^{N+1}}\Bigr)\int_0^t  |\zeta|_{L^{N+1}}^N(s) \, ds
+ \int_{\R} \int_0^t |\zeta(s)|^N \, ds \int_0^t |\zeta(s)| \, ds\nonumber \\
& \lesssim \Bigl( |u_{0,x}|_{L^{N+1}}+ |F|_{L^{2}}+ |F|_{L^{\infty}}\Bigr) \int_0^t (1+|\zeta(s)|_{L^{N+1}}^{N+1})\, ds
+t \int_0^t \int_{\R} |\zeta(s)|^{N+1} \, ds\nonumber \\
& \lesssim (1+t) \Bigl( 1+ |u_0|_{H^{\frac{3}{2}+}}^2+ |\zeta_0|_{\Lambda_{\sigma_0}}\Bigr)\Bigl(1+\int_0^t |\zeta|_{L^{N+1}}^{N+1}
(s) \, ds\Bigr)
\label{z1}
\end{align}
where in the penultimate  step we perform Holder's inequalities in time.

Finally, since $ \zeta\ge -1 $ on $[0,t] $, \eqref{uh} leads to 
$$
u_x(t) \ge  u_{0,x} - t  + F \; .
$$
Since $ N+1 $ is even, this enables to control the last term of the right-hand side member to \eqref{LM} in the following way :
\begin{eqnarray}
- \frac{N}{N+1}\int_0^t \int_{\R}  \zeta^{N+1} u_x & \le & \Bigl(t+|u_{0,x}|_{L^\infty}+|F|_{L^\infty(]0,t[\times \R)}\Bigr)
 \int_0^t \zeta^{N+1} 
\nonumber \\
& \lesssim & (1+t)  \Bigl( 1+ |u_0|_{H^{\frac{3}{2}+}}^2+ |\zeta_0|_{\Lambda_{\sigma_0}}\Bigr) \int_0^t |\zeta|_{L^{N+1}}^{N+1}\; .\label{z2}
\end{eqnarray}
Gathering \eqref{LM} and \eqref{z1}-\eqref{z2}, we infer that $ \gamma(t) =\int_0^t \int_{\R}|\zeta(s)|^{N+1} ds $
 satisfies the following differential inequality on $]0,T[ $
 \begin{equation}
\frac{d}{dt}  \gamma(t)\lesssim {|\zeta_0|_{L^{N+1}}^{N+1}}+ (1+t)  \Bigl( 1+ |u_0|_{H^{\frac{3}{2}+}}^2+ |\zeta_0|_{\Lambda_{\sigma_0}}\Bigr)
(1+\gamma(t))\; .\label{et}
 \end{equation}
Making use of  Sobolev inequalities and \eqref{boundd},  Gronwall's inequality  ensures that there exists $\tilde{C}_T>0 $ only depending on $ T $, $|u_0|_{H^{\frac{3}{2}+}}$ and 
$ |\zeta_0|_{H^{\frac{1}{2}+}} $ such that $\gamma(t) \le \tilde{C}_T $ on $ [0,T[ $. Then re-injecting this estimate in \eqref{et} we obtain 
$$
\sup_{t\in [0,T[} |\zeta(t)|_{L^{N+1}} \le C_T \; .
$$
Letting $ N\to +\infty $ this proves the  estimate on the first term in  \eqref{estest}.
Finally, the estimate on the second term in \eqref{estest} follows directly from the first one together with \eqref{uh}.\\

The continuity of the flot-map follow directly from Proposition \ref{essentiel}.
  \end{proof}
Finally we can state the following theorem as a consequence of the previous results.
\begin{theorem}\label{thmreg}
Let $s>1/2$. For $(\zeta_0,u_0) \in H^s\times H^{s+1}$ such that $1+\zeta_0 >0$, the classical Boussinesq system (\ref{1}) has a unique solution $(\zeta,u)$ in $C(\R^+,H^s\times H^{s+1} )\cap C^1(\R^+,H^{s-1}\times H^{s} ) $.\\
The  flot-map $S :(\zeta_0,u_0) \longrightarrow (\zeta,u)$ is  continuous  
 from $H^s\times H^{s+1}  $ into $ C(\R^+; H^s(\R)\times H^{s+1}(\R)) $.
\end{theorem} 

\section{Global entropy solutions  of the Boussinesq system}
In this section, we study existence of weak solution for the Boussinesq system (\ref{1}) for initial condition  $(\zeta_0,u_0) \in  \Lambda_{\sigma_0}\times H^1$. To do so, we regularize the initial data by a mollifiers sequence $ (\rho_n)_n\subset D(\R)$ by setting  $\zeta_{0,n}=\rho_n* \zeta_0 $ and $u_{0,n}=\rho_n* u_0 $, where $\rho_n(\cdot)=n\rho(n\cdot), \; \rho\in D(\R)$ such that $$0\leq \rho\leq 1, \mbox{ supp } \rho \subset [0,1] \mbox{ and } \displaystyle \int_\R \rho dx=1 .$$  
Note that $(u_{0,n})_n\subset H^s$  and it is  bounded in $H^1$ with $\|u_{0,n}\|_{H^1} \leq \|u_0\|_{H^1}$
 and $ u_{0,n} \to u_0 $ in $H^1(\R) $. For $\zeta_{0,n}$, we first notice that $ \zeta_0\in   \Lambda_{\sigma_0} $ ensures that $ 1+\zeta_0>0 $ a.e. on $ \R $. Since $ 1+\zeta_{0,n}= \rho_n * (1+\zeta_0) $, it follows that $1+\zeta_{0,n}>0 $ on $\R $. Moreover, using  (\ref{estA}) and proceeding exactly as (\ref{esta1}) by replacing  $e^{-|\cdot|}$ by $\rho_n$ it is straightforward to check that  $\zeta_{0,n} \in L^2$ with $|\zeta_{0,n}|_{L^2}\leq  c_n |\zeta_0|_{ \Lambda_{\sigma_0}}$ where $c_n$  depends on $\|\rho_n\|_{L^2}$. Similary, we can verify that $\zeta_{0,n} \in H^s$, for $s\geq 0$.\\
Now,  consider $(\zeta_n, u_n)$ the solution of (\ref{1}) emanating from $(\zeta_{0,n},u_{0,n})$ given by Proposition \ref{prop1}. we will prove that $(\zeta_n,u_n)$ has a subsequence which converges to a weak solution of the Boussinesq system (\ref{1}) with initial data $(\zeta_0,u_0)$. Note that $(\zeta_n,u_n )$ satisfies the entropy estimate (\ref{inq1}) which implies
$$
\|u_n(t)\|_{H^1}^2+ \int_{-\infty}^{+\infty}  \sigma_0(1+\zeta_n)dx 
\leq  \|u_{0,n}\|_{H^1}^2 + \int_{-\infty}^{+\infty}  \sigma_0(1+\zeta_{0,n})dx \; .
$$
The  $ H^1$-convergence of $ (u_{0,n}) $ towards $ u_0 $ ensures that the first term of the above right-hand side converges to $ \|u_0\|_{H^1}^2 $. For the second term one has to work a little more. We follow  \cite{Sch} and use the convexity of $ \sigma_0$ and Jensen inequality to get that 
$$
\sigma_0(1+\zeta_{0,n})=\sigma_0 \Bigl( \int_{\R} (1+\zeta_0)\rho_n (\cdot-z) dz\Bigr) 
\le \int_{\R} \sigma_0(1+\zeta_0) \rho_n (\cdot-z) dz \; .
$$
Therefore, integrating on $ \R $, using Fubini and $\int_{\R} \rho_n =1 $, we obtain 
$$
 \int_{-\infty}^{+\infty}  \sigma_0(1+\zeta_{0,n})dx \le
  \int_{-\infty}^{+\infty}  \sigma_0(1+\zeta_0)dx \; . 
$$
We thus are lead to the following uniform estimate on $ \R_+$ : 
\begin{equation} \label{inq2}
\|u_n(t)\|_{H^1}^2+ \int_{-\infty}^{+\infty}  \sigma_0(1+\zeta_n(t))dx 
\leq  \|u_{0}\|_{H^1}^2 + \int_{-\infty}^{+\infty}  \sigma_0(1+\zeta_{0})dx \; .
\end{equation}
 In the sequel we will make a constant use of the following lemma that can be easily deduced from \eqref{estA} and \eqref{inq2}.
\begin{lemma}\label{cor2}
let $\chi \in L^1(\R)\cap L^\infty(\R)$. Then
\begin{eqnarray}
\int_\R \left | \zeta_n (t,x)\chi(x)\right|dx \leq c,
\end{eqnarray}
for all $t$, where $c$ is independent of $n$. 
\end{lemma}
 \begin{proof} Let $M$ be defined as in (\ref{estA}). Then by \eqref{inq2}, we have
 \begin{eqnarray}
 \begin{array}{lcl}
\displaystyle \int_\R|\zeta_n(t,x)| |\chi(x)|dx&= &\displaystyle \int_{\{x/{-1<\zeta_n(t,x)|\leq M}\}} |\zeta_n(t,x)||\chi(x)|dx\\
&& \quad+\displaystyle \int_{\{x/{\zeta_n(t,x) \geq M}\}}|\zeta_n(t,x)\chi(x)|dx \\
&\leq & M \displaystyle\int_\R |\chi(x)|dx +|\chi|_{L^\infty} \int_\R \sigma_0(1+\zeta_n(t,x))dx \leq c.
\end{array}
\end{eqnarray}
\end{proof}

\begin{proposition}\label{T6} Consider the sequence $(\zeta_n,u_n) $ constructed above for $(\zeta_0,u_0)\in  \Lambda_{\sigma_0}\times H^1$ { (in particular $1+\zeta_0 >0$ {a.e. on $ \R $)}}. Then there exists a subsequence $\left((\zeta_{n_k},u_{n_k})\right)_k$ and $(\zeta,u) \in L^1_{loc}(]0,+\infty[\times \R)\times (L^\infty(]0,+\infty[,H^1(\R))\cap C(\R^+\times\R))$ such that $(\zeta_{n_k})_k$ converges weakly to $\zeta$ in $L^1$ on every compact of $]0,+\infty[ \times \R$ and $(u_{n_k})_k$ converges  weakly-$\ast$ in $L^\infty(]0,+\infty[,L^\infty (\R))$ and strongly, for any $ T>0$ , in   $C([0,T],C(\R))$ (and then in $C([0,T],L^2_{\mbox{loc}}(\R)))$ to $u$.
\end{proposition} 
{\bf{Proof.}} 
 First, applying the above lemma with $ \chi=1\!\!1_{[-A,A]} $ for $ A>0 $ we obtain that $ (\zeta_n(t))_n $ is bounded in $ L^1_{loc}(\R) $ uniformly in $ t\in\R_+ $. In particular, $(\zeta_n)_n $ is bounded in $L^1_{loc}(\R_+\times\R) $. According to Dunford-Pettis Theorem (see  \cite{DunSchw} Vol I p.294),  to prove that $(\zeta_n)_n $ is weakly compact in $ L^1(]0,T[\times ]-1,A[)$, it suffices to check that for any $ \varepsilon>0 $ there exists $ \delta>0 $ such that for any bounded measurable set 
$B\subset]0,T[\times ]-1,A[ $ with $ |B|<\delta $ it holds 
$$
\sup_{n\in\N} \int_{B} |\zeta_n|(t,x) dx\, dt < \varepsilon \; .
$$
But this follows directly from \eqref{estA} and \eqref{inq2}. Indeed, proceeding as in the proof of the above lemma, we easily check that for any $ k\ge M $ and any $B\subset]0,T[\times ]-1,A[ $,
\begin{align*}
\int_{B} |\zeta_n|(t,x) dx\, dt & \le  \int_{B\cap (-1<\xi_n\le k)}\zeta_n dx\, dt+\int_{B\cap (\zeta_n>k)}\zeta_n dx\, dt \\
& \le k |B| + (C_2^M \ln k)^{-1} \int_0^T \int_{\R} \sigma_0(1+\zeta_n) dx\, dt \\
& \le k |B| + T |\zeta_0|_{\Lambda_{\sigma_0}} (C_2^M \ln k)^{-1}\; ,
\end{align*} 
that clearly gives the desired result by taking $ k $ large enough. }

 Now, let us tackle the strong convergence of $ (u_n)_n $. By (\ref{inq2}), we have that $(u_n)_n$ is bounded in $L^\infty([0,T], H^1(\mathbb R))$. Then, $(u_n)_n$ is bounded in $L^\infty(]0,+\infty[,L^\infty (\R)).$  We deduce that it has a weakly-$\ast$  convergent subsequence in $L^\infty(]0,+\infty[,L^\infty (\R))$. Now,  we will prove that $(\partial_t u_n)_n$ is bounded in $L^\infty(]0,T[, L^2(\mathbb R))$ and after we use a theorem of Aubin-Simon to get the strong convergence. To do so, we recall that by the equation it holds 
$$ \partial_t u_n(t,x)=\int_\mathbb R k(x-z)(\zeta_n+u_n^2/2)(t,z)dz.$$
where $k(x)=\displaystyle \frac{1}{2}\sgn(x)e^{-|x|}$.
Now by (\ref{inq2}) and  Young's convolution estimates, we have
$$\int_\mathbb R k(x-z) u_n^2/2(t,z)dz=|k* u_n^2/2|^2_{L^2}\leq c |k|_{L^1} |u_n^2|_{L^2}\leq cte. $$

For the first integral on $\zeta_n$, we will use, as in \cite{Ami}, the fact that $|\sigma_0(1+\zeta_n)|_{L^1} \leq cte$ given by (\ref{inq2}) and the property of the mapping $\sigma_0(1+\cdot)$ giving in (\ref{estA}). Again denoting by $A_n(t)=\{ x\in \R/ \zeta_n(t,x)\geq M\}$,  with $ M \ge 1 $ as in \eqref{estA},
 we write
$$ \int_\mathbb R k(x-z)\zeta_n(t,z)dz=\int_{A_n^c} k(x-z)\zeta_n(t,z)dz +\int_{A_n} k(x-z)\zeta_n(t,z)dz =f_1+f_2.$$
 By  Young's convolution estimates, (\ref{estA}) and (\ref{inq2}), we get 
$$ |f_1(t,\cdot)|_{L^2}\leq |k|_{L^1} | \zeta_n 1_{A_n^c}|_{L^2}\leq  (C^M_1)^{-1/2} \left(\int_{A_n^c} \sigma_0(1+\zeta_n(t,x))dx\right)^{1/2}\lesssim |\zeta_0|_{\Lambda_{\sigma_0}}^{1/2},$$
$$ |f_2(t,\cdot)|_{L^\infty}\leq \displaystyle \frac{1}{2} \int_{A_{n}}|\zeta_n (t,x)|dx \lesssim |\zeta_0|_{\Lambda_{\sigma_0}}$$
and
$$  |f_2(t,\cdot)|_{L^1}\leq |k|_{L^1} | \zeta_n 1_{A_n}|_{L^1} \lesssim |\zeta_0|_{\Lambda_{\sigma_0}}.$$
Then, we obtain
$$  |f_2(t,\cdot)|_{L^2}\leq  |f_2|_{L^\infty}^{1/2} |f_2|_{L^1}^{1/2}\leq cte.$$
Combining the above estimates, we deduce that $\|\partial_t u_{n}\|_{L^2}$ is  bounded uniformly in $n$ and $t$. \\
Next, we prove that $(u_n)_n$ has a strongly convergent subsequence in $C([0,T]; C(\R))$, i.e. in $C([0,T]; C(K))$ for every compact $K$ of $\R$. For this end, set $K_m=[-m,m]$, by compact Sobolev injection and Aubin-Simon theorem see \cite{S87}, we have that $ E^m_{\infty,\infty}$ is compactly embedded in $ C([0,T], C(K_m))$  where $$ E^m_{\infty,\infty}= \{ u\in L^\infty(]0,T[,H^1(K_m)) \mbox{ such that }\displaystyle \partial_t u \in L^\infty(]0,T[,L^2(K_m))\}.$$ By the preceding calculations, we have proved that $(u_n)_n$ is bounded in $E^m_{\infty\infty}$. We deduce that it has a  subsequence $(u_{n_k}^m)_k$ strongly convergent in $C([0,T],C(K_m))$ $(\mbox{also in }  C([0,T],L^2(K_m)))$. By applying the diagonal extraction processus, we can construct a subsequence $(u_{n_k})_k$ which is strongly convergent in $C([0,T],C(K_m))$, for every $m\geq 1 $ and thus  in $C([0,T],C(K))$ $(\mbox{ and then in }  C([0,T],L^2(K)))$ for every compact $K$ of $\R$.
This completes the proof of the theorem.
Now, the next theorem gives a weak solution of the Boussinesq system.
\begin{theorem}\label{T7}
Let $(\zeta_n,u_n)$ be as in Proposition \ref{T6}. Then, the limit functions $(\zeta,u)$ obtained by Proposition \ref{T6} is a weak solution for the Boussinesq system with initial data $(\zeta_0,u_0)$.
\end{theorem}

\begin{proof} Let $\varphi\in C_c^\infty(]0,+\infty[\times \R)$. Multiplying (\ref{2}) by $\rho$ and integrating, we obtain 

\begin{eqnarray}\label{iT7}
\int_0^{+\infty}\int_\R\zeta_n \varphi_t dx dt+\int_0^{+\infty}\int_\R(u_n+u_n\zeta_\epsilon )\varphi_x dx dt=0
\end{eqnarray}
and
\begin{eqnarray}\label{iT8}
\int_0^{+\infty}\int_\R u_n \varphi_t dx dt+\int_0^{+\infty}\int_\R(u_n^2/2+\zeta_\epsilon)\varphi_x dx dt- \int_0^{+\infty}\int_\R u_\epsilon \varphi_{xxt} dx dt=0.
\end{eqnarray}
By taking the limit when $n$ tends to infinity, we have to prove that
$$ \int_0^{+\infty}\int_\R\zeta \varphi_t dx dt+\int_0^{+\infty}\int_\R(u+u\zeta )\varphi_x dx dt=0$$
and 
$$\int_0^{+\infty}\int_\R u \varphi_t dx dt+\int_0^{+\infty}\int_\R(u^2/2+\zeta)\varphi_x dx dt- \int_0^{+\infty}\int_\R u \varphi_{xxt} dx dt=0 $$
Since $\varphi$ is with compact support in $]0,+\infty[\times \R$ and $(\zeta_n)_n$ converges weakly in $L^1_{loc}$ to $\zeta$ we obtain
$$ \lim_{n\rightarrow +\infty}\int_0^{+\infty}\int_\R\zeta_n \varphi_tdxdt = \int_0^{+\infty}\int_\R\zeta \varphi_tdxdt.$$
And the strong convergence of $u_n$ to $u$ in $C([0,T],L^2(K))$ implies that 
$$ \lim_{n\rightarrow +\infty}\int_0^{+\infty} \int_\R (u_n-u) \varphi_x dx dt =0.$$
Let $S$ be the support of $\varphi$ and suppose that $ S \subset ]0,T[ \times ]c,d[ \subset ]0,+\infty[\times \R$. Then, we write
$$ \int_0^{+\infty}\int_\R (\zeta_n u_n - \zeta u)\varphi_xdxdt= \int_S \zeta_n(u_n-u)\varphi_x dx dt +\int_Su(\zeta_n-\zeta)\varphi_x dx dt $$
 Since  $u_n\to  u $ in  $ C([0,T];C(K)), $ for every $K$ compact of $\R$ 
and  $ (\zeta_n)_n $ is bounded in $ L^1_{loc} $ we deduce that the first term in the above right-hand side member tends to $0 $ as $n\to +\infty $. Noticing that the limit for second term  follows directly from the weak convergence of $(\zeta_n)_n$ in $ L^1_{loc} $ and the fact that $u\varphi_x \in L^\infty(S)$, we finally obtain
$$\int_0^{+\infty}\int_\R\zeta \varphi_t dx dt+\int_0^{+\infty}\int_\R(u+u\zeta )\varphi_x dx dt=0,$$
which implies that $(\zeta,u)$ satisfies the first equation of (\ref{1}) in the distribution sense.\\ 
For the second equation (\ref{iT8}), the proof is  direct  using the weak convergence of $(\zeta_n)$ and the strong convergence of $(u_n)$. \\ \\
It is still  to prove that the limit $(\zeta,u)$ satisfies the initial data $(\zeta_0,u_0).$  Recall that $(u_n) $ converges to $u$ in $C([0,T],C(K))$ (i.e. in $C([0,T]\times K))$)  for every compact $K$ of $\R$ and that  $u_{0,n}$ converges to $u_0$ in $H^1(\R)$. This enough to implies that $u(0,x)=u_0(x)$ for a.e. $x\in \R$. 
In fact we notice that 
\begin{eqnarray*}
\begin{array}{ll}
\displaystyle || (u(t,\cdot)-u_0(\cdot))||_{\infty,K} \leq   \displaystyle || (u(t,\cdot)-u_n(t,\cdot))||_{\infty,K}  & \\ 
\hspace*{3cm}+ \displaystyle || (u_n(t,\cdot)-u_{0,n}(\cdot))||_{\infty,K} +\displaystyle || (u_{0,n}(\cdot)-u_0(\cdot)))||_{\infty,K}  .&
\end{array}
\end{eqnarray*}
The above convergence results force the first and the third  term of the above right-hand side to converge towards $ 0 $ uniformly in $ t\in [0,T] $ whereas the continuity of $ u_n $ force the second term to tends to $0$ as $ t\searrow 0 $ for each fixed $n\in \N $, 
 and thus  $u(0,x)=u_0(x)$ for a.e. $x\in \R$. \\
Let us now prove that $\zeta$ satisfies the initial condition. For $\varphi \in C^\infty_c(\R)$ and $(t_k)_k \in [0,T]$ converging to $0$, we have
\begin{eqnarray}\label{ic}
\begin{array}{ll}\displaystyle \left|\int_\R (\zeta(t_k,x)-\zeta_0(x))\varphi dx\right| \leq   \displaystyle\left|\int_\R (\zeta(t_k,x)-\zeta_n(t_k,x))\varphi dx\right| &\\ 
\hspace*{2cm}+ \displaystyle \left|\int_\R (\zeta_n(t_k,x)-\zeta_{0,n}(x))\varphi dx\right|+ \left|\int_\R (\zeta_{0,n}(x)-\zeta_0(x))\varphi dx\right|. &
\end{array}
\end{eqnarray}
For the first integral, we proceed as in \cite{Sch}, Claim $4.2$  of Theorem $4.2$]. So by Dunford's lemma and  Lemma \ref{cor2}, for each $t$ there exists a subsequence $(\zeta_n^t)_n$ of $(\zeta_n)_n$ such that 
$$ \lim_{n\rightarrow +\infty} \int_\R (\zeta(t,x)-\zeta^t_n(t,x))\varphi dx =0.$$
Now applying the diagonalization process to $ (\zeta_n^{t_k})_{n,k}$, we can extract a subsequence $(\zeta_k)_k$ such that
$$\lim_{k\rightarrow +\infty} \int_\R (\zeta(t_k,x)-\zeta_k(t_k,x))\varphi dx =0.  $$ 
For the second integral, using the integral representation of $\zeta_k$, we get
\begin{eqnarray} \begin{array}{lcl} \displaystyle \int_\R (\zeta_k(t_k,x)- \zeta_{0,k}(t,x)) \varphi (x) dx&=& -\displaystyle\int_\R \int_0^{t_k} \left((u_k(s,x) +u_k \zeta_k(s,x)\right)_x \varphi(x) dsdx \\ &=&\displaystyle \int_0^{t_k} \int_\R \left((u_k(s,x) +u_k \zeta_k(s,x)\right) \varphi_x(x)dx ds .\end{array} 
\end{eqnarray}
By Lemma \ref{cor2}  it holds 
$$
\begin{array}{ll}
 \left| \displaystyle \int_0^{t_k} \int_\R \left((u_k(s,x) +u_k \zeta_k(s,x)\right) \varphi_x(x)dx ds\right| &\\ 
 \hspace*{2cm} \leq |u_k|_{L^\infty} \displaystyle \int_0^{t_k} \int_\R \left (| \varphi_x(x)| + 
|\zeta_k(s,x)\varphi_x(x)|\right) dx\leq c_1 t_k, &
\end{array}
$$
where $c_1$ is independent of $k$. So for the subsequence $ \zeta_k$, we obtain that 
$$  \displaystyle \left|\int_\R (\zeta_k(t_k,x)-\zeta_{0,k}(x))\varphi dx\right|\leq c_1 t_k. $$
The last integral in (\ref{ic}) goes to $0$ since  $ \zeta_0\in L^1_{loc} $ and thus $(\zeta_{0,k})_k $ converges to $\zeta_0$ in $L^1_{loc}$. Then, by using the subsequence $(\zeta_k)_k$ in (\ref{ic}), we deduce that
$$ \lim _{k\leftrightarrow +\infty} \displaystyle \left|\int_\R (\zeta(t_k,x)-\zeta_0(x))\varphi dx\right|=0. $$ 
\end{proof}

\section{\bf{Appendix}}
\subsection{Proof of Proposition \ref{propcom}}
Let $ N>0$.  We follow \cite{Lanlivre13}. By Plancherel and the mean-value theorem, 
 \begin{align*}
\Bigl| ( [P_N, P_{\ll N}f] g_x)(x)\Bigr| &=\Bigl|( [P_N, P_{\ll N}f] \tilde{P}_N g_x)(x)\Bigr| \\
 & =\Bigl|  \int_{\R} {\mathcal F}^{-1}_x(\varphi_N)(x-y) P_{\ll N}f(y)  \tilde{P}_N g_x(y) \, dy \\ 
 &\qquad -
 \int_{\R}  P_{\ll N}f(x)  {\mathcal F}^{-1}_x(\varphi_N)(x-y) \tilde{P}_N g_x(y) \, dy\Bigr| \\
 & = \Bigl|  \int_{\R}(P_{\ll N}f (y) -P_{\ll N}f(x)) N  {\mathcal F}^{-1}_x(\varphi)(N(x-y))  \tilde{P}_N g_x(y) \, dy\Bigr| \\
& \le \|P_{\ll N}f_x\|_{L^\infty_x}\int_{\R} N |x-y| |{\mathcal F}^{-1}_x(\varphi)(N(x-y))| |\tilde{P}_N g_x(y)| \, dy 
  \end{align*}
  Therefore, since $N |\cdot| |{\mathcal F}^{-1}_x(\varphi)(N \cdot)|=|{\mathcal F}^{-1}_x(\varphi')(N \cdot) | $ we deduce from Young's convolution and Bernstein inequalities that 
 \begin{align*} 
 \Bigl\| [P_N, P_{\ll N}f] g_x)\|_{L^2} & \lesssim  N^{-1} \|P_{\ll N}f_x\|_{L^\infty_x}  \|\tilde{P}_N g_x\|_{L^2} 
 \lesssim  \|P_{\ll N}f_x\|_{L^\infty_x}  \|\tilde{P}_N g\|_{L^2} \; .   \end{align*}
This completes the proof of estimation (\ref{cm1}). Let us prove estimation (\ref{proN1}).  Using Bernstein inequality and the characterization of the Sobolev space, we have
 \begin{equation}\nonumber
 \begin{array}{lcl}
N^s |P_N ( P_{\gtrsim N} f \, g_x)|_{L^2} &\lesssim& N^s \delta_N |P_{\gtrsim N} f \, g_x|_{L^2} \\&\leq& N^s\delta_N |P_{\gtrsim N} f |_{L^2}|g_x|_{L^\infty} \\ &\leq& \delta_N N^s(\sum_{k\gtrsim n}|P_k f|^2_{L^2})^{1/2}|g_x|_{L^\infty} \\ &\leq& \delta_N N^s(\sum_{k\gtrsim n}\delta_k^2 K^{-2s}| f|^2_{H^s})^{1/2}|g_x|_{L^\infty}\\& \leq & \delta_N N^s N^{-s} | f|_{H^s}(\sum_{k\gtrsim n}\delta_k^2 )^{1/2}|g_x|_{L^\infty}\\ &\leq& \delta_N | f |_{H^s}|g_x|_{L^\infty}.
\end{array}
 \end{equation}
Now it remains to prove 
 $$ N^s |P_N ( P_{\gtrsim N} f \, g_x)|_{L^2} \lesssim \delta_N | f |_{H^{s+1}}|g|_{L^\infty}.$$
To do, we write $P_N ( P_{\gtrsim N} f \, g_x)= \partial_xP_N ( P_{\gtrsim N} f \, g)- P_N ( P_{\gtrsim N} f_x \, g)$. The second term can be treated as above to obtain 
 $$N^s |P_N ( P_{\gtrsim N} f_x \, g)|_{L^2} \lesssim  \delta_N | f |_{H^{s+1}}|g|_{L^\infty}.$$
For the first term, we have
\begin{equation}\nonumber
 \begin{array}{lcl}
  N^s |\partial_xP_N ( P_{\gtrsim N} f \, g)|_{L^2}& \lesssim & N^s N |P_N(P_{\gtrsim N} f \, g)|_{L^2} \\& \leq & N^s N \delta_N |P_{\gtrsim N} f \, g|_{L^2} \\ &\leq & N^{s+1}  \delta_N (\sum_{k\gtrsim N}|P_k f|^2_{L^2})^{1/2}|g|_{L^\infty} \\ &\leq& N^{s+1}\delta_N (\sum_{k\gtrsim N}\delta_k^2 K^{-2(s+1)}| f|^2_{H^{s+1}})^{1/2}|g|_{L^\infty}\\ \\ &\leq & N^{s+1}\delta_N N^{-(s+1)}| f|_{H^{s+1}}(\sum_{k\gtrsim N}\delta_k^2 )^{1/2}|g|_{L^\infty}\\ &  \\ &\leq& \delta_N | f|_{H^{s+1}}|g|_{L^\infty}

 \end{array}
 \end{equation}
 \\
Finally to prove estimate (\ref{proN2}), we first notice that it follows directly from \eqref{proN1} for $ s>3/2 $ since $ H^{s-1}(\R) \hookrightarrow L^\infty(\R) $. For $ s\le 3/2 $, we start by noticing that 
$$
P_N(P_{\gtrsim N} f \, g_x) =P_N(P_{\sim N} P_{\lesssim N} g_x)  + P_N(\sum_{K\gtrsim N} P_K f  \tilde{P}_K g_x)\; .
$$
The contribution of the first term of the above  right-hand side  is easily estimated by 
\begin{align*}
N^{s-1} |P_N(P_{\sim N} f  P_{\lesssim N} g_x)|_{L^2} & \lesssim N^{s-1} |P_{\sim N} f |_{L^\infty} N^{2-s} |g|_{H^{s-1}} \\
& \lesssim N  |P_{\sim N} f |_{L^\infty} |g|_{H^{s-1}} \\
& \lesssim \delta_N  |f |_{H^{s+1}} |g|_{H^{s-1}}
\end{align*}
since $ s>1/2 $. 
On the other hand, the contribution of the second term can be estimated by 
\begin{align*}
N^{s-1} \Bigl|P_N(\sum_{K\gtrsim N} P_K f  \tilde{P}_K g_x)\Bigr|_{L^2} & 
\lesssim N^{s-1} N^{1/2}  \Bigl| P_N(\sum_{K\gtrsim N} P_K f  \tilde{P}_K g_x)\Bigr|_{L^1}\\
& \lesssim N^{s-1/2} \sum_{K\gtrsim N} 
K^{-1-s} |P_{\sim K} f |_{H^{s+1}} K^{2-s} |\tilde{P}_K g|_{H^{s-1}} \\
& \lesssim N^{s-1/2} |f |_{H^{s+1}} |g|_{H^{s-1}} \sum_{K\gtrsim N} 
K^{1-2s}\\
& \lesssim N^{1/2-s} |f |_{H^{s+1}}  |g|_{H^{s-1}}
\end{align*}
that is suitable since $s>1/2 $.

  \vspace*{2mm}  
	\subsection{Proof of Proposition \ref{propgl}}  
\begin{proof} 
$$
\begin{array}{ll}
(g_{\lambda}[\Lambda^s f],\Lambda^s f) &=\int \widehat{g_{\lambda}[\Lambda^s f]}\overline{\widehat {\Lambda^s f}}\,d\xi \\
 &= \int |\xi|^{\lambda-2}(1+\xi^2)^{s}|\xi|^2\widehat{f}\overline{\widehat{f}}\,d\xi = \int |\xi|^{\lambda-2}(1+\xi^2)^{s}\widehat{f_x}\overline{\widehat{f_x}}\,d\xi \\
&\geq \int (1+\xi^2)^{\frac{\lambda}{2}-1}(1+\xi^2)^{s}\widehat{f_x}\overline{\widehat{f_x}}\,d\xi=|f_x|^2_{H^{s-1+\lambda/2}}.
\end{array}
$$
\end{proof}
  \begin{proposition}\label{minmax}
Let $\zeta_0 \in H^\infty$ and $ \zeta \in L^\infty(]0,T[,H^\infty{\times  W^{2,1}(\R) )}$ satisfying the first equation of (\ref{2}). If $1+\zeta_0>0$ on $ \R $ then for all $t\in [0,T] $, $1+\zeta(t,x)\ge \displaystyle \min_{\R} (1+\zeta_0)  $ on $\R $.
\end{proposition}
\begin{proof} Let $ \alpha\in ]0,1[ $. We set  $\nu=1+\zeta -m_{0,\alpha}$, where $0<m_{0,\alpha} =
\displaystyle \min_{\R} (1+\zeta_0)\wedge \alpha \le \alpha< 1$. $\nu$ satisfies the equation 
\begin{eqnarray}\label{w1}
\nu_t+ (\nu u)_x+\epsilon g_\lambda(\nu)+m_{0,\alpha} u_x=0
\end{eqnarray}
Let $\nu^-=\min(0,\nu)$. Note that since for all $t\in [0,T] $, $ \zeta(t,x) \to 0 $ as $ |x|\to +\infty $ and $ \zeta\in C([0,T]; \R) $, there exists $ M>0 $ such that $ \nu^-\equiv 0 $ on $[0,T]\times (\R\setminus [-M,M]) $. This ensures that $ \nu^-\in C([0,T]; L^2)$. Multiplying by $\nu^-$ and integrating over $\R$, we get
\begin{eqnarray}\label{w2}
\displaystyle \frac{1}{2}\frac{d}{dt} \int (\nu^-)^2dx + \frac{1}{2} \int (\nu^-)^2u_x dx+ \epsilon \int g_\lambda(\nu)\nu^- dx +m_0 \int u_x \nu^- dx=0.
\end{eqnarray}
We have to prove that $\int g_\lambda(\nu) \nu^- dx\geq 0$. For this aim, set $\eta(x)=\min (0,x)^2/2$ and let $\eta_\delta=\eta*\varphi_\delta $ where $(\varphi_\delta)_{\delta} $ is a mollifiers sequence. It is easy to see that $\eta_\delta$ is a convex function of class $C^\infty(\R)$ and that  
$$ \displaystyle \int g_\lambda(\nu)\nu^- dx = \int g_\lambda(\nu) \eta'(\nu)=\lim_{\delta\rightarrow 0}\int g_\lambda(\nu) \eta_\delta'(\nu)dx,$$
using the dominated convergence theorem.  Let us check that $ \eta_\delta(\nu)\in  W^{2,1}(\R) $.
 Note that we can write 
$$ \eta_\delta (\nu)= \beta_\delta (\zeta)$$
where $\beta_\delta(\zeta)= (\eta*\varphi)(1+\zeta-m_{0,\alpha})$ and $\beta_\delta\in W^{2,\infty}$. For $\delta$ sufficiently small, it is easy to verify that $\beta_\delta(0)=0$. 
Since $ \xi \in H^\infty(\R) $, $ \xi $ is bounded on $ \R $, and since 
 $\beta_\delta \in W^{2,\infty}(I)$ for any interval $ I\subset \R $, we deduce that 
 $\eta_\delta(\nu)\in W^{2,1} $. 
  Now using Lemma \ref{lem1} we have $$ \int g_\lambda(\nu) \eta_\delta'(\nu)dx\geq \int g_\lambda (\eta_\delta(\nu))dx.$$
Since $ \eta_\delta(\nu)\in W^{2,1}$ we get that $g_\lambda(\eta_\delta(\nu))\in L^1(\R)$ and  since 
 $$
   \mathcal{F}\Bigl( g_\lambda(\eta_\delta(\nu))\Bigr)(0)=0
   $$
 this ensures that   
 $$
 \int_{\R}  g_\lambda(\eta_\delta(\nu))dx=0 \; .
 $$ 
Finally, we obtain 
$$ \displaystyle \frac{1}{2}\frac{d}{dt} \int (\nu^-)^2 dx +\frac{1}{2} \int (\nu^-)^2u_x dx +m_{0,\alpha}\int u_x \nu^- dx =- \epsilon \int g_\lambda(\nu)\nu^- dx $$
$$= \displaystyle - \epsilon\lim_{\delta\rightarrow 0} \int g_\lambda(\nu)\eta_\delta'(\nu) dx \leq- \epsilon\lim_{\delta\rightarrow 0} \int g_\lambda(\eta_\delta(\nu)) dx=0 $$
Thus, we get
$$ \displaystyle \frac{d}{dt} \int (\nu^-)^2 dx \lesssim |\nu^-|_{L^2}|u_x|_{L^\infty} + |\nu^-|_{L^2}|u_x|_{L^2} \lesssim |\nu^-|_{L^2}|u|_{s+1}.$$
By Gronwall Lemma, we have
\begin{eqnarray}\label{ineqm} \int (\nu^-)^2 dx \leq C |\nu_0^-|_{L^2}e^{\int_0^t |u|_{s+1} dt}. \end{eqnarray}
 As $\nu_0^-=0$, we deduce that $\nu \geq 0$ and then $1+\zeta\ge \displaystyle \min_{\R} (1+\zeta_0)\wedge \alpha .$
 Since it holds for any $ \alpha\in]0,1[ $, it ensures that $1+\zeta\ge \displaystyle \min_{\R} (1+\zeta_0)\wedge  1
 =\displaystyle \min_{\R} (1+\zeta_0)$.

\end{proof}
Notice that , as signaled page 15, by using the continuity of the flow map associated with $ \zeta$ ,  Proposition \ref{minmax} still valid for $s>\frac{1}{2}$, $\zeta_0 \in H^s$ and $ \zeta \in L^\infty(]0,T[,H^s)$.

\providecommand{\href}[2]{#2}


\begin{thebibliography}{10}
\bibitem{Ami}
C.~J.~Amick, Regularity and Uniqueness of Solutions for the Boussinesq System of Equations, {\it Journal of Differential Equations }{\bf 54} (1984), 231-247. 
, no.1, 49-96.

\bibitem{BBM}
T.~B.~Benjamin, J.~L.~Bona and J.~J.~Mahoney, Model equations for long 
waves in nonlinear dispersive systems, {\it Phil. Trans. Roy. Soc. London A} 
{\bf 227} (1972), 47--78.

\bibitem{BCS02}
J.~L.~Bona, M.~Chen and J.~C.~Saut, Boussinesq equations and other systems for small-amplitude long waves
in nonlinear dispersive media I: Derivation and the linear theory, {\it J. Nonlinear Sci.} {\bf 12} (2002), 283-318. 

\bibitem{BCS04}
J.~L.~Bona, M.~Chen and J.~C.~Saut, Boussinesq equations and other systems for small-amplitude long waves
in nonlinear dispersive media II: The nonlinear theory, {\it J. Nonlinearity } {\bf 17} (2004), 925-952. 

\bibitem{DunSchw}
Dunford-Schwartz, Linear Operators. PartI. Pure and Applied Mathematics. Vol VII, {\it Interscience, New York}. 
\bibitem{DGV03}
J.~Droniou, T.~ Gallouet, J.~Voyelle, Global solution and smoothing effect for a non-local regularization of a hyperbolic equation, {\it Journal of Evolution Equations }{\bf 3} (2002), 499-521.
\bibitem{DI}
J.~Droniou, C.~ Imbert, Fractal first-Order Partial Differential Equations, {\it Arch. Rational Mech. Anal. }{\bf 182} (2006),299-331.



\bibitem{Lanlivre13}
D.~Lannes. {\it The water waves problem: mathematical analysis and asymptotics.} Mathematical Surveys and Monographs (AMS) (2013).


\bibitem{Lan06}
D.~Lannes. {\it Sharp estimates for pseudo-differential operators with symbols of limited smoothness and commutators,} J. Funct. Anal. {\bf 232} (2006), 495-539.
\bibitem{SWX}
J.~C.~Saut, C.~Wang and L.~Xu, The Cauchy Problem on Large Time for Surface-Waves-Type Boussinesq Systems II,  {\it SIAM Math. Anal. } {\bf 49 (4)} (2017), 2321-2386. 
\bibitem{Sch}
M.~E.~Schonbek, Existence of Solutions to the Boussinesq System of Equations, {\it Journal of Differential Equations }{\bf 42} (1981), 325-352. 

\bibitem{S87}
J.~Simon,  compact sets in the space L$^p(0,T;\mbox{b})$, {\it  Ann. Mat. Pura appl. }{\bf 146} (1987),65-96.

\bibitem{Taylor91}
M.~I.~Taylor, {\it Tools for PDE, Pseudodifferential Operators, Paradifferential Operators, and Layer Potentials.}  Mathematical Surveys and Monographs vol. 81(AMS) (1991).

\end{thebibliography}
\end{document}